\documentclass[11pt]{amsart}

\usepackage[colorlinks=true,citecolor=black!60!green,linkcolor=black!60!red,filecolor=black!60!cyan,urlcolor=black!60!magenta]{hyperref}
\usepackage[text={6.1in,8.5in},centering]{geometry}
\usepackage{amssymb,amsmath,amsthm,mathtools,extpfeil}
\usepackage[all,cmtip]{xy}
\usepackage{tikz}
\usepackage{enumerate,multirow,array}
\usepackage{wrapfig}

\newtheorem{thm}[equation]{Theorem}
\newtheorem*{thm*}{Theorem}
\newtheorem*{lem*}{Lemma}
\newtheorem{thmA}{Theorem}

\newtheorem{lem}[equation]{Lemma}
\newtheorem{prop}[equation]{Proposition}
\newtheorem*{prop*}{Proposition}
\newtheorem{cor}[equation]{Corollary}
\newtheorem*{conj}{Conjecture}
\newtheorem*{optconj}{Optimistic conjecture}
\newtheorem*{question}{Question}

\theoremstyle{definition}
\newtheorem*{defn}{Definition}
\newtheorem{rmk}[equation]{Remark}

\numberwithin{equation}{section}

\DeclareMathOperator{\id}{id}
\DeclareMathOperator{\img}{im}
\DeclareMathOperator{\vol}{vol}

\DeclareMathOperator{\Dil}{Dil}

\DeclareMathOperator{\Map}{Map}
\DeclareMathOperator{\Lip}{Lip}
\DeclareMathOperator{\Hom}{Hom}
\DeclareMathOperator{\Hopf}{Hopf}
\newcommand{\ph}{\varphi}
\newcommand{\epsi}{\varepsilon}
\newcommand{\s}{\mathbf{S}}

\begin{document}
\title{Scalable spaces}
\author[A.~Berdnikov]{Aleksandr Berdnikov}
\address[A.~Berdnikov]{Department of Mathematics, Massachussetts Institute of Technology, Cambridge, MA, United States}
\email{aberdnik@mit.edu}
\author[F.~Manin]{Fedor Manin}
\address[F.~Manin]{Department of Mathematics, University of California, Santa Barbara, CA, United States}
\email{manin@math.ucsb.edu}
\begin{abstract}
  \emph{Scalable spaces} are simply connected compact manifolds or finite complexes whose real cohomology algebra embeds in their algebra of (flat) differential forms.  This is a rational homotopy invariant property and all scalable spaces are formal; indeed, scalability can be thought of as a metric version of formality.  They are also characterized by particularly nice behavior from the point of view of quantitative homotopy theory.  Among other results, we show that spaces which are formal but not scalable provide counterexamples to Gromov's long-standing conjecture on distortion in higher homotopy groups.
\end{abstract}
\maketitle

\section{Introduction}

Starting with the 1978 paper \cite{GrHED} and continuing in the 1990s with \cite[Ch.~7]{GrMS} and \cite{GrQHT}, Gromov has promoted the idea that the geometry of maps between simply connected spaces is governed by Sullivan's minimal models in rational homotopy theory and maps between them.    In this paper, we both show that this intuition is true in the strongest possible sense for a large class of ``scalable'' spaces, which includes spheres, complex projective spaces, Lie groups and Grassmannians, and give examples showing that the general situation is more complicated.  In particular, we disprove one of the main conjectures from \cite{GrQHT}.

Scalable spaces are closely related to formal spaces, a notion introduced by Sullivan in \cite{SulLong} and elsewhere.  A formal space is one whose rational homotopy type is a ``formal consequence'' of its rational cohomology ring; that is, all higher-order rational cohomology operations are trivial.  However, Sullivan gives two other characterizations: one in terms of quasi-isomorphisms (maps preserving cohomology) and another in terms of rational self-maps.  Scalable spaces satisfy two analogous equivalent conditions, but with a metric flavor.  In addition, they satisfy two equivalent conditions regarding Lipschitz homotopies.  The precise statement of these four equivalent conditions is given in Theorem \ref{tfae}.

\subsection{Growth, distortion, Lipschitz homotopy}
Let $X$ and $Y$ be sufficiently nice compact metric spaces, for example Riemannian manifolds or piecewise Riemannian simplicial complexes.  In \cite{GrQHT}, Gromov outlines a number of homotopical invariants concerning the asymptotic behavior of the Lipschitz constant as a functional on the mapping space $\Map(X,Y)$.  The Lipschitz constant tells us the scale at which the map becomes homotopically trivial, and therefore is a good measure of homotopical information.  Besides the inherent appeal of this program, it is important for achieving an understanding of broader questions in quantitative geometric topology, for example the questions regarding cobordism theory studied in \cite{FWPNAS} and \cite{CDMW}.

In the past decade, a fair amount of progress has been made on this program; see \cite{FWPNAS,CDMW,CMW,Guth,PCDF,IRMC,zoo,Berd}.

A convenient language for discussing asymptotics is ``big-O notation'', reviewed below:

\begin{itemize}
    \item We write $f=O(g)$ if the function $f$ grows asymptotically no faster than $g$, i.e.~$f$ is eventually bounded by the function $Cg$ for {\itshape some} $C>0$.
    \item We write $f=o(g)$ if $f$ grows asymptotically strictly more slowly than $g$, i.e.~for {\it every} $c>0$, $f$ is eventually bounded by $cf$.
    \item We write $f=\Omega(g)$ if $f$ grows asymptotically no slower than $g$, i.e.~$g=O(f)$.
    \item We write $f=\Theta (g)$ if $f$ and $g$ have the same asymptotic growth, i.e.~$f=O(g)$ and $g=O(f)$ both hold.
\end{itemize}

The most basic asymptotic invariant of $\Map(X,Y)$ is \emph{growth}: how many elements of the set of homotopy classes $[X,Y]$ have representatives with Lipschitz constant $\leq L$?  This line of inquiry goes back to \cite{GrHED}, in which Gromov proved the following:
\begin{thm*}
  For a simply connected compact Riemannian manifold $Y$, the growth of $\pi_n(Y)$ is at most polynomial in $L$.
\end{thm*}
The proof derives from rational homotopy theory.  Sullivan, following K.-T.\ Chen, had showed that all real-valued invariants $\pi_n(Y) \to \mathbb{R}$ could be computed by pulling back differential forms along a map $f:S^n \to Y$, taking wedges and antidifferentials, and finally integrating a resulting $n$-form over the sphere.  Gromov remarked that all steps of this procedure could be bounded polynomially in terms of the Lipschitz constant of the original map.

In \cite{GrQHT}, Gromov conjectured that the upper bounds on the homotopical complexity of $L$-Lipschitz maps obtained in this way are sharp.  To make this precise, it is natural to define the \emph{distortion} of an element $\alpha \in \pi_n(Y)$ to be
\[\delta_\alpha(L)=\max\{k: k\alpha\text{ has an $L$-Lipschitz representative}\}\footnote{This is essentially the inverse function of the notion used in \cite{GrQHT}, but accords with the notion of distortion used in geometric group theory.}.\]
Then Gromov's conjecture would imply that the distortion of any element is $\Theta(L^r)$ where $r$ is an integer.  Moreover, an easily stated consequence is:
\begin{conj}[Gromov]
  The distortion of an element $\alpha \in \pi_n(Y)$ is $\Theta(L^n)$ if and only if $\alpha$ has nontrivial image under the rational Hurewicz homomorphism, and $\Omega(L^{n+1})$ otherwise.
\end{conj}
The ``if'' here is easy to see using a degree argument (see Propositions~\ref{hur-bound}) and~\ref{SntoSn}; the ``only if'' has been open until now, and Gromov noted that even a proof of the first part would be remarkable.

Finally, Gromov also defined a related relative invariant: given two homotopic $L$-Lipschitz maps, we can ask for bounds on the Lipschitz constants of the intermediate maps of a homotopy.  For example, given nice compact spaces $X$ and $Y$, when can we expect two homotopic $L$-Lipschitz maps $X \to Y$ to be homotopic through $KL$-Lipschitz maps, for some constant $K=K(X,Y)$?  Ferry and Weinberger noted that for the applications they were considering, it was more useful to also bound the Lipschitz constant in the time direction.  Hence:
\begin{question}
  For what spaces $Y$ is there always a constant $K=K(X,Y)$, for any compact metric simplicial complex $X$, such that any two homotopic $L$-Lipschitz maps $X \to Y$ have a $K(L+1)$-Lipschitz homotopy?
\end{question}
Ferry and Weinberger characterized spaces satisfying a more restrictive condition, where the constant only depends on the dimension $d$ of $X$.  In that case, all homotopy groups of $Y$ must be finite.  On the other hand, it was shown in \cite{CDMW} that spaces satisfying the above condition include those that are rationally products of Eilenberg--MacLane spaces, including for example odd-dimensional spheres.  This paper also includes the first example of a target space $Y$ which does not have this property.  Moreover, in \cite{CMW} it was shown that even even-dimensional spheres do not have the property as stated; to include them in our class, we must consider only nullhomotopic maps.

A number of weaker, polynomial bounds on sizes of homotopies and nullhomotopies appear in \cite{CMW} and \cite{PCDF}, but before this paper, linearity had only been additionally proven in the case of maps $S^3 \to S^2$, by the first author \cite{Berd}.

The various quantities described here are intimately connected.  For example, in \cite{CDMW}, it is shown that if one attaches a cell along an element of $\pi_n(Y)$ with sufficiently large distortion, then the resulting complex is forced to have nonlinear nullhomotopies.  Conversely, the argument of \cite{zoo} describing the growth of $[X,Y]$ for certain $X$ and $Y$ relies on estimates on the sizes of Lipschitz homotopies.

\subsection{Main results}
The main result of this paper defines a new class of spaces in which the answers to these questions are particularly nice.
\begin{thmA} \label{tfae}
  The following are equivalent for a \emph{formal} simply connected finite complex $Y$:
  \begin{enumerate}[(i)]
  \item There is a homomorphism $H^*(Y) \to \Omega_\flat^*Y$ of differential graded algebras which sends each cohomology class to a representative of that class.  Here $\Omega_\flat^*Y$ denotes the flat forms, an algebra of not-necessarily-smooth differential forms studied by Whitney.
  \item There is a constant $C(Y)$ and infinitely many (indeed, a logarithmically dense set of) $p \in \mathbb{N}$ such that there is a $C(Y)(p+1)$-Lipschitz self-map which induces multiplication by $p^n$ on $H^n(Y;\mathbb{R})$.
  \item For all finite simplicial complexes $X$, nullhomotopic $L$-Lipschitz maps $X \to Y$ have $C(X,Y)(L+1)$-Lipschitz nullhomotopies.
  \item For all $n<\dim Y$, homotopic $L$-Lipschitz maps $S^n \to Y$ have $C(Y)(L+1)$-Lipschitz homotopies.
  \end{enumerate}
\end{thmA}
\begin{rmk}
  The conditions (i) and (ii) imply formality of $Y$ almost immediately and in fact can be seen as geometric strengthenings of two equivalent characterizations of formality given by Sullivan.  In \S\ref{S:NF}, we give an example of a non-formal space which satisfies (iv) but not (iii).  It is not clear whether (iii) implies formality.

  On the other hand, condition (i) is strictly weaker than the notion of ``geometric formality'' introduced by Kotschick \cite{Kot} based on Sullivan's observation that it is satisfied by symmetric spaces, and studied by several others.  For example, all simply connected geometrically formal 4-manifolds are rationally equivalent to $S^4$, $\mathbb{C}P^2$, or $S^2 \times S^2$.
\end{rmk}

We call spaces satisfying (i)--(iv) \emph{scalable} based on the scaling maps of (ii).  Examples of scalable spaces include spheres, projective spaces, and other symmetric spaces of compact type.  More examples of spaces known to be scalable and those known not to be scalable are given in Table \ref{table}.

We summarize some properties of scalable spaces below.
\begin{thmA}[Properties of scalable spaces] \label{props}\ 
  \begin{enumerate}[(a)]
  \item Scalability is invariant under rational homotopy equivalence.\label{Qinv}
  \item The class of scalable spaces is closed under products and wedge sums.
  \item All skeleta of scalable complexes are scalable.
  \item Scalable spaces satisfy Gromov's distortion conjecture. \label{B:dist}
  \end{enumerate}
\end{thmA}
\renewcommand{\arraystretch}{1.1}
\begin{table}
  \begin{tabular}{|l l|}
    \firsthline
    \multirow{4}{*}{
      \begin{tabular}{|l l|}
        \firsthline
        \multirow{4}{*}{
          \begin{tabular}{|l|}
            \firsthline
            \textbf{Symmetric spaces} \\
            $S^n$, $\mathbb{C}P^n$, $\mathbb{H}P^n$ \\
            Grassmannians \\
            \lasthline
          \end{tabular}
        }
        & \textbf{Scalable spaces} \\
        & $\mathbb{C}P^2 \mathbin{\#} \mathbb{C}P^2$, \\
        & $3\mathbb{C}P^2 \mathbin{\#} 3 \overline{\mathbb{C}P^2} $, \\
        & $\# r(S^n \times S^n)$, $r \leq {2n \choose n}/2$ \\[0.05cm]
        \lasthline
      \end{tabular}
    }
    & \textbf{Formal spaces} \\
    & $\#r\mathbb{C}P^2$, $r \geq 4$ \\
    & $\mathbb{C}P^n \mathbin{\#} \mathbb{C}P^n$, $n \geq 3$ \\
    & $\#r(S^n \times S^n)$, $r>{2n \choose n}/2$ \\[0.2cm]
    \lasthline
  \end{tabular}
  \caption{A Venn diagram of simply connected manifolds.}
  \label{table}
\end{table}

In fact, in Theorem \ref{MM-only} we state a stronger result related to Theorem \ref{props}\eqref{B:dist}, which essentially says that for any simply connected domain $X$, we can understand the minimal Lipschitz constant of maps in a homotopy class in $[X,Y]$ purely by looking at maps between Sullivan minimal models.  We defer the statement until it can be made precise.

On the other hand, we show that the strong form of the distortion conjecture does not always hold for non-scalable spaces, even those that are formal:
\begin{thmA} \label{grconj}
  The class of the puncture in $\pi_5([\#4\mathbb{C}P^2 \times S^2]^\circ)$
  has distortion $o(L^6)$.
\end{thmA}
We do not, however, know any matching lower bounds on distortion besides the trivial $L^5$, nor do we have upper bounds stronger than the already known $L^6$ (either of which would be very interesting.)  We merely show that the known upper bound cannot be sharp.  This is similar to the situation for Lipschitz homotopies of non-scalable formal spaces: we show that they cannot have linear Lipschitz constant, but we do not give any other lower bound for the sizes of homotopies.  This contrasts with the examples given in \cite{CDMW} and \cite{CMW}, which include an explicit lower bound.

Finally, applying Theorem \ref{tfae} to maps between wedges of spheres yields the following:
\begin{cor}
  For every rational number $r \geq 4$, there are spaces $X_r$ and $Y_r$ such
  that the growth of $[X_r,Y_r]$ is $\Theta(L^r)$.
\end{cor}
The spaces $X_r$ and $Y_r$ are constructed in \cite[Thm.~3.2]{zoo}, and it is shown there that the growth of $[X_r,Y_r]$ is asymptotically $\Omega(L^{r-\epsi})$ for every $\epsi>0$ and $O(L^r)$.  The construction of efficient maps uses a homotopy between two $O(L)$-Lipschitz maps between wedges of spheres; the remark after Example 3.1 indicates that if one had such a homotopy with Lipschitz constant $O(L)$, then the estimate on growth could be improved to $\Theta(L^r)$.  The existence of such homotopies follows from Theorem \ref{tfae}.

\subsection{Which spaces are scalable?}
To decide whether a space is scalable, we typically use condition (i) of Theorem \ref{tfae}.  To prove that a closed, formal $n$-manifold $Y$ is not scalable, we show the following local obstruction.  A map $\ph$ as in (i) sends the fundamental class $[Y]\in H^n(Y)$ to a nonzero (flat) form; in particular, it has a nonzero restriction at some point $p\in Y$.  Evaluating forms at that point, we get a mapping of graded algebras
\[\ph_p: H^*(Y;\mathbb{R})\xrightarrow{\ph} \Omega^*_{\flat}(Y)\to \bigwedge T_p^*\,,\] 
which we show to be an embedding due to the Poincar\'e pairing.  We discuss several families of manifolds for which such an embedding is impossible.  Conversely, in some cases we are able to extend a local embedding of $H^*(Y;\mathbb{R})$ in the exterior algebra on a single tangent space to an embedding into the entire $\Omega^*(Y)$.

It is tempting to conjecture that this can always be done; that is, that one can always extend an embedding of $H^*(Y;\mathbb{R})$ at one point (when $Y$ is a closed manifold) or several points (otherwise) to an embedding into $\Omega^*(Y)$.  This would imply the following additional criterion for scalability:
\begin{optconj}
  A space is scalable if and only if it is formal and $H^*(Y;\mathbb{R})$ embeds in $\bigwedge^* \mathbb{R}^N$ for some finite $N$.
\end{optconj}
Scalability would then depend only on real homotopy type---itself an open problem:
\begin{question}
  By Theorem \ref{props}(\ref{Qinv}), scalability is a rational homotopy invariant.  Is it also an $\mathbb{R}$-homotopy invariant?
\end{question}

\subsection{Scaling limits}
We now discuss techniques used in the proof of Theorem \ref{tfae}.  The most novel of these is used in showing (ii) $\Rightarrow$ (i).  Given a sequence of self-maps, we move to the sequence of induced maps $\mathcal{M}_Y^* \to \Omega^*X$ in the world of rational homotopy.  These can be formally scaled so that the corresponding geometric bounds are uniform; then by a compactness theorem, we can find an accumulation point, although this requires us to expand the space of forms to one which is complete.  This accumulation point is the map of (i).  The same technique, in combination with the algebraic impossibility discussed above, is used to prove Theorem \ref{grconj}.

These proofs are reminiscent of the work of Wenger \cite{Wenger} showing that there are nilpotent groups whose Dehn function is not exactly polynomial.  There the role of the limiting object obtained after scaling is played by the asymptotic cone, and one can use the algebraic structure of the nilpotent group to prove the nonexistence of a filling with certain bounds.  Since nilpotent groups can also be studied using rational homotopy theory, it would be interesting to get a stronger handle on the formal similarities between these arguments.

\subsection{Formal maps to genuine maps}
The \emph{shadowing principle} introduced in \cite{PCDF} allows formal, rational homotopy-theoretic maps and homotopies to be upgraded to actual maps between spaces with only a linear deterioration in geometric bounds.  This finds a number of applications in this paper; the simplest is (i) $\Rightarrow$ (ii) of Theorem \ref{tfae}.

A more involved application of the shadowing principle is the direction (ii) $\Rightarrow$ (iii).  This is a generalization of the first author's argument \cite{Berd} proving that maps $S^3 \to S^2$ have linear nullhomotopies, which we summarize as follows.    A map $S^3 \to S^2$ can be ``regularized'' via a short homotopy to have a nice structure: imagine a bowl of spaghetti, in which the sauce is mapped to the basepoint of $S^2$, while the cross-section of each noodle maps homeomorphically onto its complement.  The construction iteratively ``combs'' the spaghetti at larger and larger scales: $2$, $4$, and so on up to $2^{\log(\Lip f)}$.  Each step takes twice as long as the previous one, but there are logarithmically many steps total, making for a linear bound.  Finally the last map is well-organized enough to be nullhomotoped by hand.

Here, we generalize this idea by abstracting the components: a ``combed'' $O(L)$-Lipschitz map is the composition of an $O(L/K)$-Lipschitz map with a ``scaling'' self-map as provided by (ii).  The shadowing principle is used to produce both the intermediate maps (by ``squinting at'' the original map) and the homotopies between them.

\subsection{How to read this paper}

The first few sections are intended to provide examples of most of the phenomena discussed in this paper without requiring knowledge of rational homotopy theory.  Section \ref{S:flat} introduces flat differential forms, which are an important technical tool throughout. Section \ref{S3} proves some simple facts about linear algebra which allow us to show that certain spaces are not scalable.  Section \ref{S4} introduces the basic techniques of quantitative homotopy theory and gives examples of some of the phenomena which occur in non-scalable spaces, one of which is the proof of Theorem \ref{grconj}.  In Section \ref{S:scal}, we show that certain high-dimensional manifolds are scalable, beyond the obvious examples of symmetric spaces and their wedges and products.

In Section \ref{S2}, we discuss rational homotopy theory and its relationship to quantitative results, introducing necessary facts from \cite{PCDF} and relating them to flat forms.   The remaining sections all use this material in an essential way.  The reader who is interested in a slower-paced introduction to the subject is invited to consult \cite{PCDF} for a treatment focusing on quantitative results or a textbook on the subject such as \cite{GrMo}.  Section \ref{S:NF} discusses an example which demonstrates that our methods don't extend straightforwardly to non-formal spaces.  Finally, Section \ref{S5} gives the proof of Theorem \ref{tfae} and Section \ref{S:B} gives the proof of Theorem \ref{props}; one particularly technical point
is banished to an additional final section.

\subsection{Acknowledgements}

We would like to thank Robert Young, who pointed out the reference \cite{Wenger}, as well as Robin Elliott and Shmuel Weinberger for other useful comments.  We also thank the anonymous referee for a large number of corrections as well as clarifying questions and remarks which greatly improved the exposition and led us to several discoveries.  The second author was supported by the NSF via the individual grant DMS-2001042. 

\section{Flat differential forms} \label{S:flat}

For technical reasons we need to introduce the flat differential forms on $X$, denoted by $\Omega_\flat^*(X)$.  These can be defined in several ways:
\begin{itemize}
\item As the dual normed space to the space of flat chains on $X$ \cite[\S IX.7]{GIT}.
\item As the set of $L^\infty$ forms with $L^\infty$ differential, cf.~\cite[Thm.~1.5]{GKS}.  Here the differential of a non-smooth form is defined using Stokes' theorem applied to its action on currents.
\item As the set of (non-smooth) differential forms satisfying certain complicated ``niceness'' conditions, see \cite[\S IX.6]{GIT}.
\end{itemize}
We also write $\Omega_\flat^*(X,A)$ to denote the subalgebra of flat forms that are identically zero on a subcomplex $A$.

Flat forms have a number of attractive properties:
\begin{lem}[{see \cite[\S3]{GKS}}] \label{flatQI}
  The inclusion $\Omega^*(X) \to \Omega_\flat^*(X)$ induces an isomorphism on cohomology.
\end{lem}
\begin{lem}[{see \cite[Theorem 3.6]{GKS}}] \label{flatLip}
  Flat forms pull back to flat forms along Lipschitz maps.
\end{lem}

A sequence of flat forms is said to \emph{weak$^\flat$ converge} if its values on every flat chain converge (this is an instance of weak$^*$ convergence.)
\begin{lem} \label{lem:dga}
  Weak$^\flat$ limits commute with $d$ and $\wedge$.
\end{lem}
\begin{proof}
  The former is true by definition and the latter is shown in \cite[\S IX.17]{GIT}.
\end{proof}

Finally, we need a version for flat forms of a result originally stated by Gromov and proved among other places as \cite[Lemma 2--2]{PCDF}:
\begin{lem}[Coisoperimetric inequality] \label{coIP}
  Let $A \subset X$ be a simplicial pair with a linear metric.  For every $k$ there is a constant $C(k,X,A)$ such that every exact form $\omega \in \Omega_\flat^{k+1}(X,A)$ has a \emph{primitive} (an $\alpha \in \Omega_\flat^k(X,A)$ satisfying $d\alpha=\omega$) such that $\lVert\alpha\rVert_\infty \leq C(k,X,A)\lVert\omega\rVert_\infty$.
\end{lem}
The proof of \cite[Lemma 2--2]{PCDF} holds verbatim for flat forms once one defines fiberwise integration for these.  This can be done either directly using the $L^\infty$ definition, or by defining a dual notion of shadows of flat chains.

\section{Obstructions to scalability} \label{S3}

In order to show that some spaces aren't scalable we use the following test. 

\begin{prop} \label{LocRep}
  If $X$ is a scalable closed $n$-manifold, then there is an embedding of graded algebras 
  \[H^*(X;\mathbb{R})\hookrightarrow \Lambda^* \mathbb{R}^n.\]
\end{prop} 

\begin{proof}
By property (i) of scalable spaces, $X$ admits a presentation of its cohomology algebra via flat forms:
\[\ph: H^*(X;\mathbb{R}) \hookrightarrow \Omega^*_\flat X.\]
Instead of dealing with the whole algebra $\Omega^*_\flat X$, we focus just on the values of these points at some point $p\in X$: 
\[\ph_p: \alpha \mapsto \ph(\alpha)(p) \in \Lambda (T_p^* X) \cong \Lambda^* \mathbb{R}^n.\]
Since $X$ is a simply connected (and therefore orientable) closed manifold, one can consider a fundamental class $[X]\neq 0\in H^n(X)$ and chose the point $p$ from outside of the zero locus of $\ph([X])$.  Then whenever one has $a\smile b =[X]$ in cohomology, the value $\ph_p(a)$ has to be non-zero, since $\ph_p(a)\wedge \ph_p(b)=\ph_p([X])\neq 0$. And since the pairing of $H^i (X)$ and $H^{n-i}(X)$ is non-singular by Poincar\'e duality, no $a\in H^i(X)$ could go to $0\in \Lambda^i T^*_p$ under $\ph_p$, thus $\ph_p$ provides the required embedding.

There is a slight technical wrinkle here in that flat forms are only defined up to a measure zero set.  To make sense of the previous paragraph, one needs to choose representatives; the equation $\ph(a) \wedge \ph(b)=\ph([X])$ will then be true on a set of full measure for each choice of $a$ and $b$ in a basis for the cohomology.  Since $\ph([X])$ is nonzero on a set of positive measure, we can find a point $p$ at which $\ph_p([X]) \neq 0$ and all the equations are satisfied.
\end{proof}

The criterion we just proved allows to easily rule out scalability for some manifolds by simply comparing the ranks of $H^*(X)$ and $\Lambda(\mathbb{R}^n)$:
\begin{cor} \label{nScal1}
  If $X$ is a scalable closed $n$-manifold, then the rank of $H^k(X;\mathbb{R})$ is at most ${n \choose k}$.  In particular, $\# r(S^n \times S^n)$ is not scalable for $r>{2n \choose n}/2$, and $\# r(S^n \times S^m)$ is not scalable for $r>{m+n \choose n}$.
\end{cor}
This restriction only holds for closed manifolds; for example, arbitrary wedges of spheres and manifold thickenings thereof are scalable.

Next, we point out some slightly more subtle reasons that certain cohomology algebras cannot be embedded in the alternating algebra $\bigwedge^*V$ for any finite-dimensional $\mathbb{R}$-vector space $V$.
\begin{thm} \label{subalg}
  The graded algebras listed below cannot be embedded in $\bigwedge^*V$ for any $V=\mathbb{R}^N$. The degree $n$ of a generator $x$ is shown by a superscript as in $x^{(n)}$. 
  \begin{enumerate}[(i)]
  \item For all $n \geq 1$, the algebra
    \[\Omega_{n,r}=\bigl\langle a_i^{(n)},b_i^{(n)}\:(1 \leq i \leq r) \mid
    \:\forall i,j:a_ib_i=a_jb_j, a_ia_j=b_ib_j=0;\ 
    \:\forall i \neq j: a_ib_j=0\bigr\rangle\]
    for $r>\frac{1}{2}{2n \choose n}$.  (On the other hand,
    $\Omega_{n,\frac{1}{2}{2n \choose n}}$ embeds in $\bigwedge^*\mathbb{R}^{2n}$.)
  \item For all even $n \geq 2$, the algebra
    \[\Sigma_{n,r}=\bigl\langle a_i^{(n)}\:(1 \leq i \leq r) \mid
    \: \forall i \neq j: a_i^2=a_j^2, a_ia_j=0\bigr\rangle\]
    for all $r>\frac{1}{2}{2n \choose n}$.  (On the other hand,
    $\Sigma_{n,\frac{1}{2}{2n \choose n}}$ embeds in $\bigwedge^*\mathbb{R}^{2n}$.)
  \item For all $n \geq 3$, the algebra
    \[\Pi_{n,r}=\bigl\langle a_i^{(2)}\:(1 \leq i \leq r) \mid
    \: \forall i \neq j: a_i^n=a_j^n, a_ia_j=0 \bigr\rangle\]
    for all $r>1$.
  \end{enumerate}
\end{thm}
\begin{cor} \label{nScal}
  The following spaces are not scalable:
  \begin{enumerate}[(i)]
  \item $p\mathbb{C}P^2 \mathbin{\#} q \overline{\mathbb{C}P^2}$ when either $p>3$ or $q>3$.
  \item $p\mathbb{H}P^2 \mathbin{\#} q\overline{\mathbb{H}P^2}$ when either $p>35$ or $q>35$.
  \item $p\mathbb{O}P^2 \mathbin{\#} q\overline{\mathbb{O}P^2}$ when either $p>6435$ or $q>6435$.
  \item $\# r \mathbb{C}P^n$ for $n \geq 3$ and $r>1$.
  \end{enumerate}
\end{cor}
\begin{proof}
  In all the cases, as above, we can restrict an embedding in $\bigwedge^*\mathbb{R}^N$ to a subspace $\mathbb{R}^{2n} \subset \mathbb{R}^N$ on which the top class is nontrivial.  Moreover, this restriction is still an embedding since each of the algebras satisfies Poincar\'e duality, in the sense that its multiplication defines a bilinear pairing between elements of degree $k$ and degree $2n-k$.

  \subsubsection*{Case (i):}
  As mentioned above, if $r>{2n \choose n}/2$, the number of $n$-dimensional generators is greater than the dimension of $\bigwedge^n \mathbb{R}^{2n}$, and therefore an embedding cannot exist.

  Conversely, suppose that $r={2n \choose n}/2$ and let $\mathbb{R}^{2n}$ be generated by $x_1,\ldots,x_{2n}$.  Then we can assign the generators to the ${2n \choose n}$ degree $n$ monomials generated by $dx_1,\ldots,dx_{2n}$, with $a_i$ and $b_i$ assigned to complementary choices.

  \subsubsection*{Case (ii):}
  Again suppose $\mathbb{R}^{2n}$ is generated by $x_1,\ldots,x_{2n}$, and fix a volume form $dx_1 \wedge \cdots \wedge dx_{2n}$.  Then $\wedge$ induces a symmetric bilinear form on $\bigwedge^n\mathbb{R}^{2n}$ of signature $\bigl({2n \choose n}/2,{2n \choose n}/2\bigr)$, with basis vectors
  \[dx_I+dx_{I^c} \quad and \quad dx_I-dx_{I^c}\]
  squaring to $1$ and $-1$ respectively.  Here $I$ is a choice of $n$ indices between $1$ and $2n$ and $I^c$ is its positively oriented complement.  Then we can assign ${2n \choose n}/2$ generators to forms of the form $dx_I+dx_{I^c}$.  On the other hand, if $r>{2n \choose n}/2$, then an assignment of these generators would imply the existence of a basis in which the bilinear form has $I_r$ as a minor, which cannot happen.

  \subsubsection*{Case (iii):}
  Assume that $n \geq 3$.  We would like to show that there cannot be two symplectic forms on $\mathbb{R}^{2n}$ whose wedge product is zero.  For some basis $x_1,\ldots,x_{2n}$, one of these is
  \[\omega=dx_1 \wedge dx_2+dx_3 \wedge dx_4+\cdots+dx_{2n-1} \wedge dx_{2n}\]
  and the other one is $\eta=\sum_{i<j} w_{ij}dx_i \wedge dx_j$ for some coefficients $w_{ij}$. For convenience, denote $w_{(2i-1)(2i)}$ by $u_i$. Then the condition $\omega \wedge \eta=0$ is a system of linear equations of the form
  \begin{align*}
    w_{k\ell} &= 0, &\text{for }w_{k\ell} \text{ other than }u_i,\\
    u_i+u_j &= 0, & \text{for }i \neq j.
  \end{align*}
  The only potentially non-zero coefficients are $u_i$, but even they vanish if $n\geqslant 3$: in that case $u_1=-u_2=u_3=-u_1$, and same goes for any $u_i$.  Thus $\eta=0$.
\end{proof}

\section{Phenomena in non-scalable spaces} \label{S4}

In this section we give two examples in non-scalable spaces in which a rescaling and convergence argument gives new asymptotic lower bounds on the Lipschitz constant of maps.  The first is a special case of Theorem \ref{tfae} but is proven using a more direct method in order to demonstrate the technique.  The second is a counterexample to Gromov's distortion conjecture.

While these examples can be understood perhaps more elegantly via maps from minimal models, we have chosen to make them accessible without any knowledge of rational homotopy theory.

\subsection{First examples of Lipschitz bounds}
We first go over some methods of establishing relationships between the Lipschitz constant and homotopy class of a map which date back to \cite{GrHED}.
\begin{prop} \label{pullback-bound}
  Let $X$ and $Y$ be Riemannian manifolds with boundary, and suppose that $f:X \to Y$ is an $L$-Lipschitz map.  Then for any flat form $\omega \in \Omega^n_\flat(Y)$,
  \[\lVert f^*\omega \rVert_\infty \leq L^n \lVert \omega \rVert_\infty.\]
\end{prop}
\begin{proof}
  If $f$ is a smooth map and $\omega$ is a smooth form, then $\lVert f^*\omega \rVert_\infty$ is the supremal value of $f^*\omega$ over $n$-frames of vectors of length at most 1.  The pushforwards of these vectors have length at most $L$, so this supremum is bounded above by the supremal value of $\omega$ on $n$-frames of vectors of length at most $L$.  This is $L^n\lVert \omega \rVert_\infty$.
  
  In the general case, one can think of the $\infty$-norm as dual to the mass norm on flat chains, and use the fact that pushing forward by $f$ multiplies the mass of a flat $n$-chain by at most $L^n$ \cite[4.1.14]{Fed}.  This duality argument is used in \cite[\S X.8]{GIT} for flat forms in open subsets of $\mathbb{R}^n$, but it works in any space in which the definitions make sense.
\end{proof}

In particular, Proposition \ref{pullback-bound} restricts the action of an $L$-Lipschitz map on cohomology. That is the main tool of bounding the homotopy class of an $L$-Lipschitz map. Here is the most straightforward application:
\begin{cor} \label{hur-bound}
   If $\alpha \in \pi_n(Y)$ has nontrivial image under the rational Hurewicz homomorphism, the distortion of $\alpha$ is $O(L^n)$.  That is, the Lipschitz constant of any representative of $k\alpha$ is $\Omega(k^{1/n})$.
\end{cor}
\begin{proof}
  By assumption, $H^n(Y,\mathbb{Q})\xrightarrow{\alpha^*} H^n(S^n,\mathbb{Q})$ is nonzero, so there is a form $\omega \in \Omega^*(Y)$ such that for every representative $f:S^n \to Y$ of $\alpha$ $\int f^*(\omega)=v\neq 0$. Thus, using Proposition~\ref{pullback-bound} we can bound the degree $k$ of a representative $g$ of $k\alpha$ by
  \[\vol(S^n)\lVert g^*\omega \rVert_\infty \leqslant \vol(S^n)(\Lip g)^n \lVert \omega \rVert_\infty. \qedhere\]
\end{proof}
For example, when applied to a volume form on $S^n$, this proposition implies that the degree of an $L$-Lipschitz map $S^n \to S^n$ is at most $L^n$. 
Gromov observed that this estimate is sharp up to a multiplicative constant:
\begin{prop} \label{SntoSn}
  For every $d$, there is a map $S^n \to S^n$ of degree $d$ whose Lipschitz constant is $C(n)d^{1/n}$.
\end{prop}
\begin{proof}
  Let $\ell$ be smallest integer such that $\ell^n \geq d$; then $\ell \leq 2d^{1/n}$.  Give $S^n$ the metric of $\partial [0,1]^{n+1}$, which is bilipschitz to the round metric, and divide one of the faces into $\ell^n$ identical sub-cubes, $\ell$ to a side.  We map all other faces to a base point, and map $d$ of the sub-cubes to the sphere by a rescaling of a standard degree 1 map
  \[f:([0,1]^n,\partial[0,1]^n) \to (S^n,\text{pt}).\]
  The resulting map has degree $d$ and Lipschitz constant $\ell\Lip f \leq (2 \Lip f)d^{1/n}$.
\end{proof}

\subsection{The Hopf invariant}

Gromov's next example in \cite{GrHED} concerns maps $S^3 \to S^2$.  The group $\pi_3(S^2)$ is isomorphic to $\mathbb{Z}$; this isomorphism is realized by the \emph{Hopf invariant} of a map $f:S^3 \to S^2$, which can be computed in several ways:
\begin{itemize}
  \item Hopf's original definition: if $f$ is a smooth map, its Hopf invariant is the linking number between the preimages of any two regular values $p$ and $q$.
  \item J.~H.~C.~Whitehead's formula: if $f$ is a smooth map, its Hopf invariant is given by
  \[\Hopf(f)=\int_{S^3} \alpha \wedge f^*\omega\]
  where $\omega \in \Omega^2(S^2)$ is any closed form which integrates to 1 and $\alpha$ is any primitive for $f^*\omega$, which is a closed $2$-form in $S^3$ and therefore exact.  (Indeed, $\alpha$ can even be a primitive for $f^*\omega'$, for some $\omega'$ in the same cohomology class as $\omega$.)
  \item The algebraic topology definition: build a $4$-complex $X_f$ by attaching a $4$-cell to $S^2$ using $f$.  This has cohomology classes $w$ generating $H^2(S^2)$ and $b$ generating $H^4(X_f)$.  Then the Hopf invariant of $f$ is the number $\operatorname{Hopf}(f)$ such that
  \[w^2=\Hopf(f)b.\]
\end{itemize}
It is not hard to see the relationships between these definitions.  In Whitehead's formula one can choose $\omega$ and $\omega'$ to be concentrated near $p$ and $q$, respectively, and then choose $\alpha$ to be concentrated near an oriented surface filling $f^{-1}(q)$.  This shows that Hopf's definition is a special case of Whitehead's.

On the other hand, Whitehead's formula is related to the algebraic topology definition via Stokes' theorem.  The attaching map $f:S^3 \to S^2$ extends to a map $F:D^4 \to X_f$, and $f^*\omega$ extends to a $2$-form $\tilde\omega \in \Omega^2(D^4)$.  Then by Stokes' theorem,
\[\Hopf(f)=\int_{S^3} \alpha \wedge f^*\omega=\int_{D^4} \tilde\omega^2.\]

The Hopf map $h:S^3 \to S^2$ is the map of Hopf invariant $1$ which is the attaching map of the top cell of $\mathbb CP^2$.  Gromov's argument shows that the Hopf map $S^3 \to S^2$ has distortion $\Theta(L^4)$.  The upper bound uses Whitehead's formula.  By Proposition \ref{pullback-bound}, if $f$ is $L$-Lipschitz, then $\lVert f^*d\vol_{S^2} \rVert_\infty \leq L^2$, and by Lemma \ref{coIP} we can also choose $\alpha$ so that $\lVert \alpha \rVert_\infty \leq CL^2$.  Therefore
\[\operatorname{Hopf}(f) \leq \vol(S^3)\lVert\alpha\rVert_\infty\lVert f^*d\vol_{S^2} \rVert_\infty \leq CL^4.\]
Therefore the distortion of the Hopf map, whose Hopf invariant is 1, is $O(L^4)$.  Conversely, one can build an $O(L)$-Lipschitz map with Hopf-invariant $L^4$ by postcomposing $h$ with the degree $L^2$ map $f_L:S^2 \to S^2$ of Proposition \ref{SntoSn}.  For a regular value $p$, $(f_Lh)^{-1}(p)$ consists of $L^2$ circles, each linking with every other one.  This gives a total self-linking number of $L^4$.

\subsection{Homotopy classes of attaching maps}

The examples in this section use generalizations of the Hopf invariant.  We use all three definitions: Whitehead-like homotopy invariants will be derived using Stokes' theorem and used to produce geometric bounds.  To generate intuition for these bounds, one should pretend that the forms are Poincar\'e dual to submanifolds, as in Hopf's definition, and measure the thickness with which those submanifolds can be embedded.
\begin{figure}
  \begin{center}
  \begin{tikzpicture}[scale=1.8]
    \draw[gray,thick,-stealth] (0,-0.2) -- (0,2) node[anchor=west]{$y$};
    \draw[gray,thick,-stealth] (-0.2,0) -- (2,0) node[anchor=north]{$x$};
    \draw[line width=1.5pt] (-0.25,1) -- (2.05,1) node[anchor=north]{$\alpha$};
    \draw[line width=1.5pt] (1,-0.25) -- (1,2.05) node[anchor=east]{$\beta$};
    \node[inner sep=2.5pt,circle,draw=black,fill=black,thick] at (1,1) {};
    \node[anchor=north east] at (1,1) {$\alpha \wedge \beta$};
    \draw[line width=2.5pt] (1,1) .. controls (1.2,1.4) and (1.9,1.3) .. (2.05,1.6) node[anchor=south east]{$\gamma$};
  \end{tikzpicture}
  \end{center}
  \caption{In this example, $\alpha$ and $\beta$ are 1-forms of the form $f(x)dx$ and $f(y)dy$ (where $f$ is a bump function), and $\alpha\wedge\beta=d\gamma$.}\label{fig:dgamma}
\end{figure}
We suggest the following dictionary for the reader's marginal doodles:
\begin{itemize}
    \item Cup products correspond to intersections.
    \item Differentials correspond to boundaries of submanifolds (and hence primitives correspond to fillings).
    \item The $\infty$-norm of a form corresponds to the thickness of the embedded normal neighborhood of the Poincar\'e dual submanifold.
\end{itemize}
See Figure \ref{fig:dgamma} for an illustration.

Let $M$ be a smooth oriented $n$-manifold which is not a rational homology sphere, that is the fundamental cohomology class is a nontrivial cup product.  This implies that if $M$ is punctured, the puncture generates a nontrivial class in $\pi_{n-1}$, since the cup product structure distinguishes $M$ from $M^\circ \vee S^n$.  We can also think of this class as the top-dimensional attaching map in a cell structure for $M$ with one $n$-cell.

Let $\omega$ and $\eta$ be two forms of complementary dimension on $M$ such that $\int_M \omega \wedge \eta=1$, and let $f:S^{n-1} \to M^{(n-1)}$ be a map to the $(n-1)$-skeleton of $M$.  Then
\[I(f)=\int_{S^{n-1}} \alpha \wedge f^*\eta,\]
where $\alpha$ is any primitive for $f^*\omega$, is a homotopy invariant of $f$.  By Stokes' theorem, if $f$ is in the homotopy class of $N$ times the attaching map of the top cell, then $I(f)=N$.

\subsection{Nonlinear homotopies}
\begin{thm} \label{CP2sharp4}
  There is no $C>0$ such that every nullhomotopic $L$-Lipschitz mapping $S^3 \to \# 4\mathbb{C}P^2$ has a $CL$-Lipschitz nullhomotopy.
\end{thm}
We first note that $\#4 \mathbb{C}P^2$ can be given a CW structure with four $2$-cells corresponding to the copies of $\mathbb{C}P^1$ inside each $\mathbb{C}P^2$, together with one top cell whose attaching map in $\pi_3(\bigvee_4 S^2)$ is the sum of the elements corresponding to the Hopf fibration over each of the spheres.
\begin{proof}
  We start with a specific family of maps $f_N:S^3 \to \bigvee_4 S^2\subset \# 4\mathbb{C}P^2$ which are $C_0N$-Lipschitz; we will show by way of contradiction that there is no $C_1$ such that each $f_N$ extends to a $C_1N$-Lipschitz map $D^4 \to \# 4\mathbb{C}P^2$.

  Let $S_i$ be a copy of $\mathbb{C}P^1$ inside the $i$th copy of $\mathbb{C}P^2$.  We define the $f_N$ so that they map the outside of four fixed balls $B_1,\ldots,B_4$ to the basepoint $*$.  On each $B_i$, $f_N$ maps to $S_i$ with Hopf invariant $N^4$; specifically, as a composition
  \[B^3 \xrightarrow{\text{Hopf map}} S_i \xrightarrow{\text{degree }N^2} S_i,\]
  where the degree $N^2$ map has homeomorphic preimages of $S_i \setminus *$ lined up in a square grid within a fixed square inside $S^2$, similar to the construction in Lemma \ref{SntoSn}.  The maps $f_N$ are nullhomotopic since they are homotopic in $\bigvee_4 S^2$ to $N^4$ times the attaching map of the 4-cell.
  
  Suppose now that, for some $C_1$, every $f_N$ extends to a $C_1N$-Lipschitz map $h_N:D^4 \to \# 4\mathbb{C}P^2$.  Let $\alpha_i$ be forms with disjoint support Poincar\'e dual to the cycles represented by $S_i$; then for each $i$, $\alpha_i^2$ is a representative of the fundamental class of $\# 4\mathbb{C}P^2$, so let $\gamma_1,\gamma_2,\gamma_3 \in \Omega^3(\# 4\mathbb{C}P^2)$ satisfy $d\gamma_i=\alpha_i^2-\alpha_4^2$.  By Proposition \ref{pullback-bound}, $h_N$ changes the sup-norm of $k$-forms by a factor of at most $(C_1 N)^k$.  In particular, the scaled pullback forms $a_{i,N}=\frac{1}{N^2}h^*_N\alpha_i$ and $g_{i,N}=\frac{1}{N^4}h_N^*\gamma_i$ satisfy:
  \begin{align*}
    \lVert a_{i,N}\rVert_\infty &\leqslant \frac{1}{N^2}(C_1N)^2\lVert \alpha_i\rVert_\infty=C_1^2\lVert \alpha_i\rVert_\infty \\
    \lVert g_{i,N}\rVert_\infty &\leqslant \frac{1}{N^4}(C_1N)^3\lVert \gamma_i\rVert_\infty\xrightarrow{N \to \infty} 0.
  \end{align*}
  Thus, we can chose a subsequence $(N_k)$ so that $a_{i,N_k}$ and $g_{i,N_k}$ weak$^\flat$-converge to some $a_{i,\infty}$ and 0, respectively. Moreover, since $a_{i,N}^2-a_{i,N}^2=d g_{i,N}$ and weak$^\flat$ limits commute with $\wedge$ and $d$,  this means that $a_{i,\infty}^2=a_{4,\infty}^2$ for each $i$.

  On the other hand, by Stokes' theorem,
  \[\int_{D^4} (h_N^*\alpha_i)^2=\int_{S^3} f_N^*\alpha_i \wedge \eta\]
  where $\eta$ is a form satisfying $d\eta=f_N^*\alpha_1|_{S^3}$; that is, this integral is the Hopf invariant of the projection of $f_N$ to $S_1$.  Therefore, $\int_{D^4} a_{1,\infty}^2=1$; in particular $a_{1,\infty}^2$ is nonzero at some point.

  This means that we have constructed an embedding $H^*(\# 4\mathbb{C}P^2;\mathbb{R}) \to \bigwedge^* \mathbb{R}^4$; but by Corollary \ref{nScal}, this cannot exist.
\end{proof}

\subsection{Proof of Theorem \ref{grconj}}
\begin{thm}
  The distortion of the generator $\alpha \in \pi_5([(\# 4\mathbb{C}P^2) \times S^2]^\circ)$ is $o(L^6)$.
\end{thm}
This disproves the strong form of Gromov's conjecture (that any element with trivial Hurewicz image in $\pi_n(Y)$ has distortion $\Omega(L^{n+1})$) and in particular shows that not all formal spaces satisfy the conjecture.
\begin{proof}
  Write $Y=[(\# 4\mathbb{C}P^2) \times S^2]^\circ$.  We use an argument very similar to the previous one.  Take a purported sequence of $CN$-Lipschitz maps $f_N:S^5 \to Y$ representing $N^6\alpha$.

  Let $\alpha_1,\ldots,\alpha_4$ and $\beta$ be forms on $Y$ pulled back (along the natural projections) from standard generators of $H^2(\# 4 \mathbb{C}P^2)$ and $H^2(S^2)$ respectively. We may assume that the $\alpha_i$ have disjoint support and that $\alpha_i^2 \wedge \beta=0$ (for example by pulling back our original choice along the deformation retraction of $Y$ to its 4-skeleton.)  Finally, as before, we define $\gamma_1$, $\gamma_2$, $\gamma_3$ such that $d\gamma_i=\alpha_i^2-\alpha_4^2$. 
  
  Again by Proposition \ref{pullback-bound}, the $CN$-Lipschitz maps $f_N$ change the sup-norm of $k$-forms by a factor of at most $(CN)^k$. That implies that the norms of the forms
  \[\frac{1}{N^2}f^*_N\alpha_i, \quad \frac{1}{N^2}f_N^*\beta, \quad
  \frac{1}{N^4}f_N^*\gamma_i\]
  are bounded, and moreover the $\frac{1}{N^4}f_N^*\gamma_i$ converge to 0. So we can choose a sequence of $N$ such that these forms converge to some $\alpha_i^\infty$, $\beta^\infty$, and $0$, respectively.

  Moreover, the $\alpha_i^\infty$ are nonzero, by the following reasoning.  Let $\eta_N \in \Omega^1(S^5)$ be such that $d\eta_N=f_N^*\beta$.  Then
  \begin{equation} \label{Stokes}
    \int_{S^5} f_N^*\alpha_i^2 \wedge \eta_N=N^6.
  \end{equation}
  This can be seen by Stokes' theorem, as follows.  The $\alpha_i$ and $\beta$ can be extended to forms $\hat\alpha_i$ and $\hat\beta$ over the unpunctured $\hat Y=(\# 4\mathbb{C}P^2) \times S^2$, retaining all their properties except the vanishing of $\alpha_i^2\wedge \beta$; instead, $\hat\alpha_i^2 \wedge \hat\beta$ represents the fundamental class of $\hat Y$.  Let $h_N$ be a nullhomotopy of $f_N$ in $\hat Y$; we know that this nullhomotopy must have degree $N^6$ over the puncture, and therefore $\int_{D^6} h_N^*\hat\alpha_i^2 \wedge \hat\beta=N^6$.  By Stokes' theorem, \eqref{Stokes} holds.

  On the other hand, by the coisoperimetric inequality Lemma \ref{coIP}, we can take $\eta_N$ so that $\lVert\eta_N\rVert_\infty \lesssim N^2$; this allows us to choose a further subsequence in which the $N^{-2}\eta_N$ converge weakly to some $\eta^\infty$ with $d\eta^\infty=\beta^\infty$.  Moreover, since $\wedge$ commutes with weak limits, $\int_{S^5} (\alpha^\infty_i)^2 \wedge \eta^\infty=1$. Therefore, $\int_{S^5} \lvert (\alpha^\infty_i)^2 \rvert_\infty d\vol \gtrsim 1$, and in particular $(\alpha^\infty_i)^2$ is nonzero.

  In other words, $(\alpha^\infty_i)^2=(\alpha^\infty_j)^2 \neq 0$ for every $i$ and $j$, but $\alpha_i^\infty \wedge \alpha_j^\infty=0$ for every $i \neq j$.  By Theorem \ref{subalg} (case (ii), $n=2, r=4$) this cannot happen locally at any point.
\end{proof}

\section{Examples of scalable spaces} \label{S:scal}
In this section we prove that certain connected sums are in fact scalable by showing that they have the condition~(i) listed in Theorem~\ref{tfae}.  The basic idea is to  use Poincar\'e duality, building forms supported on the tubular neighborhoods of certain submanifolds.
\begin{thm} \label{scnm}
  For any $n\leq m$ and $r \leq {n+m-1 \choose n-1}$ the space
  $\# r(S^n \times S^m)$ is scalable.
\end{thm}
\noindent In particular, once we combine this result with Prop.~\ref{nScal1}, we know the exact cutoff for scalability for spaces of the form $\# r(S^n \times S^n)$; for $m \neq n$ there remains a gap.  One corollary is as follows:
\begin{cor}
  The following spaces are scalable:
  \begin{itemize}
  \item $p\mathbb{C}P^2 \mathbin{\#} q\overline{\mathbb{C}P^2}$, $0 \leq p,q \leq 3$.
  \item $p\mathbb{H}P^2 \mathbin{\#} q\overline{\mathbb{H}P^2}$, $0 \leq p,q \leq 35$.
  \item $p\mathbb{O}P^2 \mathbin{\#} q\overline{\mathbb{O}P^2}$, $0 \leq p,q \leq 6435$.
  \end{itemize}
\end{cor}
\begin{proof}
  We start by ``symmetrizing'' $\# r(S^n \times S^n)$ by adding $(n+1)$-cells which make the two factors of each $S^n \times S^n$ homotopic to each other. The resulting space $\Sigma_{n,r}$ is still scalable because the inclusion map
  \[\# r(S^n \times S^n) \hookrightarrow \Sigma_{n,r}\]
  induces an injection on cohomology, and the corresponding forms are easy to extend over the additional cells.

  Recall that formal spaces, which will be discussed in more detail in the next section, have a rational homotopy type that is a ``formal consequence'' of their rational cohomology algebra.  In particular, two formal spaces that have the same rational cohomology algebra are rationally equivalent, as is the case for $\#r\mathbb{C}P^2$ and $\Sigma_{2,r}$.  Thus by Theorem \ref{props}(\ref{Qinv}), $\#2\mathbb{C}P^2$ and $\#3 \mathbb{C}P^2$ are scalable.

  Similarly, $\mathbb{C}P^2 \mathbin{\#} \overline{\mathbb{C}P^2}$ is formal and has the same rational cohomology algebra as $S^2 \times S^2$.  More generally, $p\mathbb{C}P^2 \mathbin{\#} q\overline{\mathbb{C}P^2}$ is rationally equivalent to $\# \max(p,q)(S^2 \times S^2)$ with $|q-p|$ of the connected summands ``symmetrized''.

  The quaternionic and octonionic cases are similar.
\end{proof}
On the other hand, our results say nothing about ``mixed'' connected sums such as $(S^2 \times S^2) \mathbin{\#} \mathbb{C}P^2$, since while their real cohomology algebras are isomorphic to ones we understand, their rational cohomology algebras are not.  If we knew that scalability is a real homotopy invariant, we would understand it for all simply connected 4-manifolds.

One can make a more general statement than Theorem \ref{scnm} to treat the case of different summands, although the condition on $r$ becomes a bit convoluted.  We say a family $\mathcal{I}$ of subsets $I \subset \{0,1,\dots,k\}$ is \emph{intersection-complete} if for any $I,J\in \mathcal{I}$ all four intersections of $I$ or $I^c$ with $J$ or $J^c$ are non-empty.

\begin{thm} \label{sci}
  For any intersection-complete family $\mathcal{I}$ of subsets $I\subset \{0,1,\dots,k\}$, the following space is scalable:
  \[X_{\mathcal{I}}:=\#_{I\in\mathcal{I}} \bigl(S^{|I|} \times S^{k+1-|I|}\bigr).\]
\end{thm}
We use the notation $S_I$ and $S_{I^c}$ for the two spheres in the summand corresponding to $I$.

Theorem~\ref{scnm} is recovered from this statement by choosing as $\mathcal{I}$ any subcollection of the $\binom{n+m-1}{n-1}$ subsets of $\{0,\dots, n+m-1\}$ of cardinality $n$ that contain 0.  This, together with the inequality $n\leq m$, ensures that the family is intersection-complete.

\begin{proof}
  We start by explaining why the combinatorial formulation makes sense. To show that condition~(i) from Theorem~\ref{tfae} holds we need to present the cohomology ring $H^*(X_{\mathcal{I}})$ by forms $\omega \in \Omega^*_\flat(X_{\mathcal{I}})$.  The space $X_{\mathcal{I}}$ has a simple cell decomposition: it is a disk $D^{k+1}$ attached to a wedge of spheres~$\bigvee_I S_I$ by the sum of Whitehead products $[\id_{S_I},\id_{S_{I^c}}]$.
  
  So we start building these forms near the center of the disk $D^{k+1}\subset \mathbb{R}^{k+1}$ by sending the generator of $H^{|I|}(S_I)$ to the form $\omega_{I}:=\bigwedge_{i\in I} dx_i$, for each $I \in \mathcal{I} \cup \mathcal{I}^c$ (where $\mathcal{I}^c$ represents the set of all complements of elements of $\mathcal{I}$, not the complement of $\mathcal{I}$).  The fact that the family is intersection-complete then implies that any two such forms have a common $dx_i$ and hence multiply to $0$, unless they are $\omega_I \wedge \omega_{I^c}=\bigwedge_i dx_i$.  This way we get a multiplicative structure isomorphic to that of $H^*(X_{\mathcal{I}})$.
  
  Now it remains to extend the forms $\omega_I$ from a region $[-1,1]^{k+1} \varsubsetneq D^{k+1}$ to the rest of the disk so that on the boundary $\partial D^{k+1}$ they turn out to be pullbacks along the attaching map of the volume forms on the spheres $S_I$.  We summarize this in the following lemma, which we take the rest of the section to prove.
  \begin{lem} \label{coordforms}
    For any intersection-complete family $\mathcal{I}$ of subsets of $\{0,\ldots,k\}$, the forms $\omega_I$, $I\in \mathcal{I} \cup \mathcal{I}^c$, can be extended to closed forms on $D^{k+1}$ so that
    \[\omega_I|_{\partial D^{k+1}}=f^*\alpha_I,\]
    where the forms $\alpha_I$ are the volume forms of
    \[S_I \subset \bigvee_{I\in\mathcal I} (S_I \vee S_{I^c}) \subset X_{\mathcal I},\]
    and $f$ is the previously mentioned attaching map, and such that the product of the forms is zero outside $[-1,1]^{k+1}$.\qedhere
  \end{lem}
\end{proof}

\subsection{Proof of lemma~\ref{coordforms}}
Let's overview the rough idea of the construction.  First we extend the forms $\omega_I$ from the cube $[-1,1]^{k+1}$ to a much larger cube via the same formula
\[\omega_I=\bigwedge_{i \in I} \chi_{x_i\in[-1,1]}dx_i.\]
In other words, on any large sphere around the origin, $\omega_{I^c}$ is concentrated near, and Poincar\'e dual to, the coordinate sphere
\[S(I):=\{\mathbf{x}\in S^{k} \mid x_i=0, i\in I^c\}.\]
Taking $\overline{\mathcal{I}}$ to be the closure of $\mathcal{I} \cup \mathcal{I}^c$ under intersections, the coordinate spheres $S(J)$ for $J \in \overline{\mathcal{I}} \setminus \emptyset$ are the closed strata of a stratification of $S^k$.  (By a stratification here, we mean a collection of submanifolds $X(I) \subset X$ associated to elements $I$ of a semilattice $\mathcal I$, where $\overline{X(I)} \cap \overline{X(J)}=\overline{X(I \cap J)}$, and with conditions on the neighborhood of a point of $X(I)$.  We don't need to make the concept completely precise because the general notion of such a stratification is only used to build intuition.)

Outside this cube, we will construct forms that are again more or less Poincar\'e dual to closed submanifolds in a stratification of an annular region in $D^{k+1}$, which we equate to $S^k \times [0,T]$.  This stratification restricts to the stratification by coordinate spheres on $S^k \times \{0\}$, and the strata indexed by all $J$ outside of $\mathcal{I} \cup \mathcal{I}^c$ have trivial intersection with $S^k \times \{T\}$.

We describe this as a kind of stratified framed bordism, that is we examine the intersections of the strata with concentric spheres centered at the origin and describe their evolution as ``time'', i.e.~radius, increases.  Over time, the strata are ``peeled off'' one by one, starting with the maximal ones.  These maximal strata are stored aside after being detached, while all subsequent lower ones are peeled off and then collapsed. 

Each time a stratum departs, however, it leaves behind a small part of itself, concentated near and held in place by lower strata.  We reinterpret the leftover pieces as data associated to fibers over the lower strata: here we actually keep track of the forms rather than the strata themselves. Luckily, the exact shapes that are added this way don't matter, as all the lower-dimensional strata eventually collapse.  But we do use the fact that they are globally almost products, in a sense which we now describe.

\begin{defn}
  A \emph{thickening} of the stratification by coordinate spheres described above is determined by a choice of numbers $1<\epsi_I\ll \text{Rad}(S^k)$ for every $I \in \overline{\mathcal I}$ which satisfy $\epsi_J \gg \epsi_I$ whenever $J \subset I$.  Then the \emph{(closed) membrane} $\s_I$ is defined to be the $\epsi_I$-neighborhood of the coordinate sphere $S(I)$.  The \emph{open membrane} $\s^{\circ}_I$ is $\s_I\setminus \bigcup\limits_{J \varsubsetneq I}\s_J$.
\end{defn}
To start, we must pick the initial $\epsi_I$'s small enough that we can pass to significantly thicker membranes a number of times over the course of the argument.

The membrane $\s_I$ is canonically diffeomorphic to $S(I) \times D^{k+1-|I|}$, with coordinates $(x,r,\theta)$ representing the point at distance $r$ along the geodesic from $x \in S(I)$ to $\theta \in S(I^c)$.  We say that a form $\omega$ \emph{agrees with} our thickening if on any open membrane
\[\s^{\circ}_{I}\cong S(I) \times D(\epsi_I) \setminus \bigcup_{J \varsubsetneq I}\s_J\]
it depends only on the $D(\epsi_I)$ coordinates, i.e., $\omega$ is the pullback of some $\omega_D\in \Omega^*(D(\epsi_I))$ under the projection to the second factor.

This notion of agreement is crucial for the description of the construction, so it will be maintained throughout the rest of the section.

As the procedure consists of peeling off the membranes, we need to specify a way to detach them.

\begin{defn}
  Given a thickening and closed forms $\omega_{J,0}$ which agree with it, a \emph{pinching off} of a membrane $\s_I$ in the direction of $p_I\in S(I^c)$ is a new thickening (with new forms $\omega_{J,1}$) such that:
  \begin{enumerate}
  \item All the change in the forms is supported in a small neighborhood of $\s_I$, and outside all $\s_J$ for $J \neq I$.
  \item The membrane $\s_I$ is replaced by a parallel thickened sphere $\s'_I$ which is shifted slightly in the direction of $p_I$ and doesn't cross any other membranes.  This carries forms that agree with its product structure.  The new thickening does not have a membrane corresponding to $I$.
  \item For $J \subset I$, the $\s_J$ are thickened in a consistent way, and the forms are changed in such a way that they agree with the new thickening; for other $J$, the thickening does not change.
  \item The forms $\omega_{J,0}$ and $\omega_{J,1}$ extend to closed forms $\omega_{J,t} \in \Omega^*(S^k \times [0,1])$ whose pairwise products are still zero.
\end{enumerate}
\end{defn}

\begin{figure}
  \includegraphics[width=\textwidth]{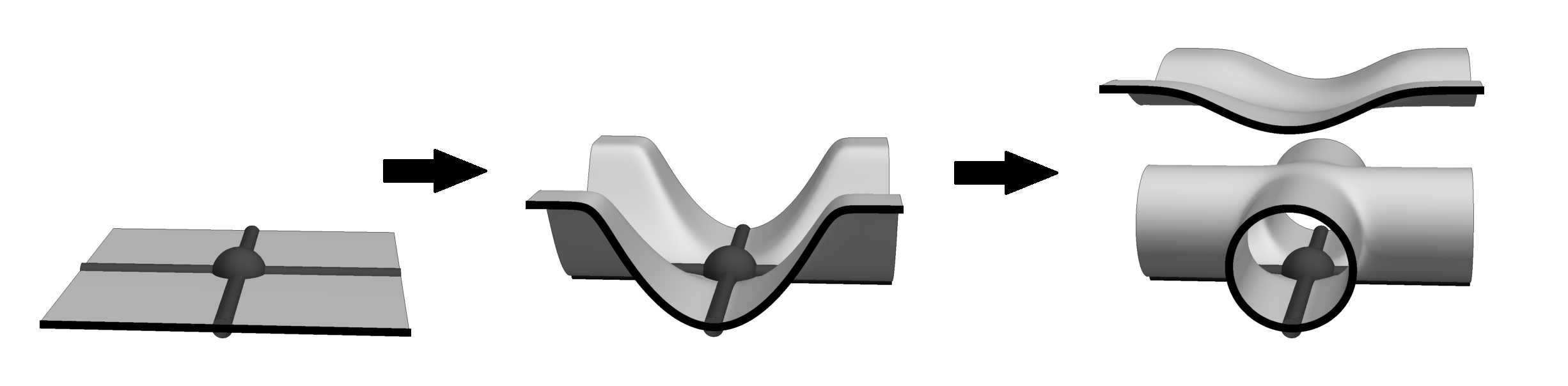}

  \caption{An example of pinching off a 2-membrane (gray) with 1- and 0-membranes (black) standing still.  The 2-membrane leaves behind some tubes that are inherited by the lower membranes and are incorporated into them.}  \label{tubes}
\end{figure}

\begin{lem} \label{pinch}
  Let $\mathcal J$ be a set of subsets of $[0,\ldots,k]$ which is closed under intersection.  For a sufficiently thin thickening of the stratification of $S^k$ by coordinate spheres $S(J)$, $J \in \mathcal J$, any maximal membrane $\s_I$ can be pinched off in any direction $p_I$ such that the shift of $\s_I$ in that direction wouldn't intersect other membranes.
\end{lem}

This is the main technical lemma, but its proof is just a wordy description of Figure \ref{tubes}. So we put it in a separate section~\ref{pinchproof}. With this tool at hand, we are ready to prove Lemma~\ref{coordforms}.

The forms on the cylinder $S^k\times[0,T]$ are constructed on cylinders $S^k\times[t,t+1]$, one after the other, half of which are pinch-off cylinders.

First we pinch off the maximal membranes. Since $\mathcal{I}$ is intersection-complete, no two sets of indices in $\mathcal I \cup \mathcal I^c$ are contained in one another. Therefore the supports of the forms $\omega_I$ are precisely the maximal membranes.  Given a maximal membrane $\s_I$, pick a point $p_I \in S(I^c)$ that is far away from any lower membranes.  That ensures that the geodesic disk $D$ with center $p_I$ and boundary $S(I)$ only intersects the open membranes $\s_{J\subseteq I}^\circ$ and $\s_{I^c}^\circ$.  So we pinch off $\s_I$ in the direction of $p_I$ and contract it along $D$ to be a tight loop around $p_I$.  Then we pinch off $\s_{I^c}$ in a similar way and contract it along an analogous disk to be tightly linked with the new $\s_I$.

After every maximal membrane is dealt with in this way, we have a set of Whitehead-linked spheres, as required by the conclusion of the lemma.  We move them all to a small ball so they can be ignored for the rest of the construction; it remains to kill the remaining membranes.

This is done inductively, from the top down.  Any now-maximal membrane $\s_I$ can be pinched off in the direction of the $p_J$ that was picked earlier for one of the original maximal membranes $\s_J$, $J \supset I$.  The geodesic disk $D$ with center $p_J$ and boundary $S(I)$ only intersects membranes near its boundary, since we have already gotten rid of $\s_{J^c}^\circ$.  Thus after pinching off, we can extend the resulting sphere to a disk in $S^k \times [t,t+1]$ which does not intersect any other membranes.  The normal bundle extends to the trivial bundle on this disk, so we can extend the forms to
ones on $S^k \times [t,t+1]$ which agree with this bundle structure.  On the
remaining thickened stratification, the forms do not change on this interval.

After collapsing all the lower-order membranes, all that remains is the linked
spheres, as required by the statement of the lemma. \qedhere

\subsection{Pinch-off lemma}

\label{pinchproof}

It remains to prove Lemma~\ref{pinch}.  Recall that the aim is to pinch off a maximal membrane $\s_I$ so that it splits into a parallel disjoint sphere $\s_I'$, plus some leftovers that are brushed under the set $\bigcup_{J \subsetneq I} \s_J$. First, observe that on a neighborhood $K$ of $\s_I$ (which includes $\s_I'$) we can choose coordinates
\begin{equation} \label{eq:prod}
  S(I) \times D^{k-|I|} \times [-\epsi, \delta+\epsi],
\end{equation}
which preserve trivializations, such that $\s_I \subset S(I) \times D^{k-|I|} \times [-\epsi,\epsi]$ and such that the last coordinate represents roughly the direction from $S(I)$ to $p_I$, which is angled away from other membranes intersecting $S(I)$.  In particular, we can arrange it so that $\s_I'$ is just $\s_I$ shifted by $\delta$ in the direction of the last coordinate, and $S(I) \times D^{k-|I|} \times [\epsi, \delta+\epsi]$ does not intersect any membranes.

We now define forms on $K \times [0,1]$ which extend the $\omega_J$ on $K \times \{0\}$ and are time-invariant on the $\s_{J'}$, $J' \subsetneq I$.  Recall that on $\s_I^\circ$, the $\omega_J$ are independent of the sphere coordinate, that is $\omega_J|_{\s_I^\circ}$ is the pullback of a compactly supported form $\alpha_J \in \Omega^*(D^{k+1-|I|})$ along the projection to the disk coordinate.  Let $K^\circ=K \setminus \bigcup_{J' \subsetneq I} \s_{J'}$.  We define $\omega_J$ on $K^\circ \times [0,1]$ via $\omega_J=(\pi \times \tau)^*\alpha_J$, where $\pi(x,y,r,t)=y$ is the projection to $D^{k-|I|}$ and
\[\tau:S(I) \times D^{k-|I|} \times [-\epsi,\delta+\epsi] \times [0,1] \to [-\epsi,\delta+\epsi]\]
is a Lipschitz function satisfying:
\begin{enumerate}[(a)]
\item $\tau(x,y,r,0)=r$ for all $x$ and $y$ (so that $\omega_J|_{K \times \{0\}}$ is as desired).
\item For $t \in (0,1)$, $\tau(x,y,r,t)$ interpolates linearly between $r$ and $\tau(x,y,r,1)$.
\item $\tau(x,y,r,1)$ depends only on $r$ and the distance $d$ from $x$ to the set $\bigcup_{J' \subsetneq I} S(J')$.  Moreover, as a function of $r$ and $d$ it is the piecewise linear function described by Figure \ref{fig:tau}.
\end{enumerate}
\begin{figure}
  \begin{tikzpicture}%
	[z={(0cm, 1cm)},
	x={(-0.8cm, -0.3cm)},
	y={(-0.7cm, 0.35cm)},
	scale=0.8]
  	
  	\filldraw[very thick,fill=gray!50!white] (0,-0.5,-0.5) -- (0,0.5,0.5) -- (0.6,0.5,0.5) -- (0.6,-0.5,-0.5) -- cycle;
  	\draw[-stealth] (-1,0,0.3) node[anchor=west,align=center] {$\bigcup_{J' \subsetneq I} \s_{J'}$ fits in here\\(and $\tau$ is independent of $t$\\on this subdomain)} -- (-0.1,0,0);
	
	\foreach \x in {3,3.2,...,4} {
	  \draw (\x,-1,-1)--(\x,5,-1)--(\x,7,1);
	}
	\foreach \x in {0.2,0.4,...,1.8} {
	  \draw (3-\x,-1,-1) -- (3-\x,1-\x,-1) -- (3-\x,1,-1+\x) -- (3-\x,1+\x,-1) -- (3-\x,5,-1) -- (3-\x,7,1);
	}
	\foreach \x in {0,0.2,...,1} {
	  \draw (\x,-1,-1)--(\x,1,1)--(\x,3,-1)--(\x,5,-1)--(\x,7,1);
	}
	\foreach \y in {-1,3,4,5} {
  	  \draw (0,\y,-1) -- (4,\y,-1);
  	}
	\draw (0,0,0) -- (1,0,0) -- (2,0,-1) -- (4,0,-1);
	\draw (0,1,1) -- (1,1,1) -- (3,1,-1) -- (4,1,-1);
	\draw (0,2,0) -- (1,2,0) -- (2,2,-1) -- (4,2,-1);
  	\draw (0,6,0) -- (4,6,0);
  	\draw (0,7,1) -- (4,7,1);
  	
  	\node[anchor=north west] at (3,-1,-1) {$d=\epsi$};
  	\node[anchor=north west] at (0,-1,-1) {$d=0$};
  	\node[anchor=north east] at (4,-1,-1) {$r=-\epsi$};
  	\node[anchor=north east] at (4,1,-1) {$\epsi$};
  	\node[anchor=north east] at (4,3,-1) {$3\epsi$};
  	\node[anchor=north east] at (4,5,-1) {$\delta-\epsi$};
  	\node[anchor=north east] at (4,7,1) {$r=\delta+\epsi$};
  \end{tikzpicture}
  \caption{A graph of $\tau(x,y,r,1)$ as a function of $r$ and $d$.  It is not to scale, but faithfully shows the linear pieces.  The range is $[-\epsi,\epsi]$.} \label{fig:tau}
\end{figure}
The membrane pinched off in Figure \ref{tubes} is the preimage under $\tau$ of a small neighborhood of zero.  In Figure \ref{fig:tau}, we see that at time $t=1$, this preimage splits into a shifted copy $\s_I'$ and a bubble surrounding the lower strata.  We expand all the $\s_J$, $J \subsetneq I$, so that they have radius between $3\epsi$ and $\delta-\epsi$ and therefore encompass this bubble.  This can be done as long as $\delta$ is sufficiently larger than $\epsi$.

We argue that the new forms at time 1 agree with this new thickening.  Indeed, if $(x,y,r)$ is contained in the newly thickened $\s_J^\circ$, then the closest point to $x \in S(I)$ in the lower strata is in some $S(J')$ such that $J' \supseteq J$.  Then the values of the forms at $(x,y,r,1)$ only depend on the distance from $x$ to this point and on $y$ and $r$.  All of these depend only on the fiber coordinate in $\s_J$.  On the other hand, outside all of the $\s_J$, the forms are zero except on $\s_I'$.

Finally, on $\s_I'$, $\tau$ is independent of $x$, and therefore $\omega_I$ agrees with the product structure \eqref{eq:prod}. \qedhere

\section{Rational homotopy theory} \label{S2}

In this section we introduce Sullivan's formulation of rational homotopy theory using differential forms, emphasizing the quantitative aspects outlined in \cite{PCDF}.  We also explain why these results apply to flat as well as smooth forms.  We recommend the book by Griffiths and Morgan \cite{GrMo} for a more thorough introduction to the algebraic aspects.

The basic category of Sullivan's theory is that of differential graded algebras (DGAs).  A DGA is a chain complex over a field $\mathbb{F}$ (in our case, always $\mathbb{Q}$ or $\mathbb{R}$) equipped with a graded-commutative multiplication satisfying the (graded) Leibniz rule.  The prototypical examples are:
\begin{itemize}
\item The smooth forms $\Omega^*(X)$ on a smooth manifold $X$, or the simplexwise smooth forms on a simplicial complex.
\item Sullivan's \emph{minimal DGA} $\mathcal{M}_Y^*(\mathbb{F})$ for a simply connected space $Y$, which is a free algebra generated in degree $n$ by the \emph{indecomposable} elements $V_n=\Hom(\pi_n(Y);\mathbb{F})$ and with a differential determined by the $k$-invariants in the Postnikov tower of $Y$.  We will write $\mathcal{M}_Y^*$ to mean $\mathcal{M}_Y^*(\mathbb{R})$.
\end{itemize}
The cohomology of a DGA is the cohomology of the underlying chain complex.  The correct notion of an equivalence between DGAs is a \emph{quasi-isomorphism}, a map which induces an isomorphism on cohomology.  In particular, for every simply connected manifold or simplicial complex $Y$ there is a quasi-isomorphism, which we call the \emph{minimal model},
\[m_Y:\mathcal{M}_Y^* \to \Omega^*(Y),\]
constructed by induction on the indecomposable elements of $\mathcal{M}_Y^*$.

When $Y$ is compact, $\Omega^*(Y)$ is finite-dimensional and $\mathcal{M}_Y^*$ is finitely generated in every degree; so a reductionist perspective is that $m_Y$ is simply a choice of a finite number of forms on $Y$ satisfying certain relations.  Nevertheless, the perspective of shifting between maps $f:X \to Y$ and homomorphisms $\ph:\mathcal{M}_Y^* \to \Omega^*(X)$ via the correspondence
\[f \mapsto f^*m_Y\]
turns out to be quite powerful.  We think of homomorphisms $\mathcal{M}_Y^* \to \Omega^*(X)$, most of which are not images of genuine maps under this correspondence, as ``formal'' maps from $X$ to $Y$.

\subsection{Flat forms and minimal models}
Here we demonstrate the advantages of using flat forms $\Omega^*_\flat (X)$ rather than smooth forms to define our minimal models.  First, Lemmas \ref{flatQI} and \ref{flatLip} imply the following:
\begin{itemize}
\item Any minimal model for $\Omega^*(X)$ is also a minimal model for $\Omega_\flat^*(X)$.
\item Any minimal model $m_Y:\mathcal{M}_Y^* \to \Omega^*_\flat(Y)$ induces an algebraicization map $f \mapsto f^*m_Y$ sending
  \[\{\text{Lipschitz maps }X \to Y\} \to \Hom(\mathcal{M}_Y^*,\Omega^*_\flat(X)).\]
\end{itemize}

Given finite complexes $X$ and $Y$, we define a weak$^\flat$ topology on $\Hom(\mathcal{M}_Y^*,\Omega_\flat^*(X))$ generated by the topologies on the restrictions to each indecomposable.  In other words, a sequence of maps converges if and only if it converges on every indecomposable.
\begin{lem}
  A sequence of maps in $\Hom(\mathcal{M}_Y^*,\Omega_\flat^*(X))$ whose $L^\infty$ norm on each indecomposable is bounded has a weak$^\flat$-convergent subsequence.
\end{lem}
\begin{proof}
  We note that this also bounds the flat norm on each indecomposable, since the differential is generated by indecomposables in lower degrees.  By the Banach--Alaoglu theorem, the restriction of the sequence to every indecomposable has a weak$^\flat$-convergent subsequence.  Since we can choose a finite basis of indecomposables of degree $\leq \dim X$, this gives us a subsequence which weak$^\flat$-converges on all indecomposables.  By Lemma \ref{lem:dga}, this subsequence in fact converges to a DGA homomorphism.
\end{proof}

Together with Lemma \ref{coIP}, these observations are enough to show that the machinery of \cite{PCDF} still works when we substitute flat forms for smooth ones.

\subsection{Obstruction theory}

Classical obstruction theory describes the obstruction to constructing a lifting-extension
\[\xymatrix{
  A \ar[r] \ar@{^(->}[d] & Y \ar[d]^p \\
  X \ar[r] \ar@{-->}[ru] & B,
}\]
where $p:Y \to B$ is a principal fibration with fiber $K(\pi,n)$: this obstruction lies in $H^{n+1}(X,A;\pi)$ and is obtained by pulling back the $k$-invariant of the fibration, which lies in $H^{n+1}(B;\pi)$.  In particular, it is often fruitful to take $Y$ and $B$ to be adjacent stages of the Postnikov tower of a space.

A similar obstruction theory for minimal DGAs is described in \cite[\S10]{GrMo}.  In this case, the role of a Postnikov stage of $\mathcal M_Y^*$ is played by the sub-DGA $\mathcal{M}_Y^*(k)$ generated by indecomposables of degree $\leq k$.  More generally, the obstruction theory can be stated for \emph{elementary extensions} $\mathcal A \to \mathcal A\langle V \rangle$, where $V$ is a vector space of indecomposables in degree $k$ with differentials in $\mathcal A^{n+1}$.

To give the precise statements, we must introduce other ideas.  First define homotopy of DGA homomorphisms as follows: $f,g:\mathcal{A} \to \mathcal{B}$ are homotopic if there is a homomorphism
\[H:\mathcal{A} \to \mathcal{B}\otimes\mathbb{R}\langle t,dt \rangle,\]
where $t$ is of degree zero, such that $H|_{\substack{t=0\\dt=0}}=f$ and $H|_{\substack{t=1\\dt=0}}=g$.  We think of $\mathbb{R}\langle t,dt \rangle$ as an algebraic model for the unit interval and this notion as an abstraction of the map induced by an ordinary smooth homotopy.  In particular, it defines an equivalence relation \cite[Cor.~10.7]{GrMo}.  Moreover, for any piecewise smooth space $X$ there is a map
\[\rho:\Omega^*_\flat(X) \otimes \mathbb{R}\langle t,dt \rangle \to
\Omega_\flat^*(X \times [0,1])\]
given by ``realizing'' this interval, that is, interpreting the $t$ and $dt$ the way one would as forms on the interval.  When the target is a space $\Omega^*_\flat(X)$ of flat forms, we can obtain the same homotopy theory by replacing DGA homotopies with ``formal'' homotopies with images in $\Omega^*_\flat(X \times [0,T])$.  We often use these two notions of homotopy interchangeably in this paper.

To help define concrete DGA homotopies, we introduce some extra notation.  For any DGA $\mathcal A$, define an operator $\int_0^t:\mathcal A \otimes \mathbb{R}\langle t,dt \rangle \to \mathcal A \otimes \mathbb{R}\langle t,dt \rangle$ by
\[{\textstyle\int_0^t a \otimes t^i}=0, {\textstyle\int_0^t a \otimes t^idt}=(-1)^{\deg a}a \otimes \frac{t^{i+1}}{i+1}\]
and an operator $\int_0^1:\mathcal A \otimes \mathbb{R}\langle t,dt \rangle \to \mathcal A$ by
\[{\textstyle\int_0^1 a \otimes t^i}=0,{\textstyle\int_0^1 a \otimes t^idt}=
(-1)^{\deg a}\frac{a}{i+1}.\]
These provide a formal analogue of fiberwise integration; in particular, they satisfy the identities
\begin{align}
  d\bigl({\textstyle\int_0^t} u\bigr)+{\textstyle\int_0^t} du &=
  u-u|_{\substack{t=0\\dt=0}} \otimes 1 \label{eqn:I0t} \\
  d\bigl({\textstyle\int_0^1} u\bigr)+{\textstyle\int_0^1} du &=
  u|_{\substack{t=1\\dt=0}}-u|_{\substack{t=0\\dt=0}}. \label{eqn:I01}
\end{align}

The relative cohomology of a DGA homomorphism $\ph:\mathcal{A} \to \mathcal{B}$ is the cohomology of the cochain complex
\[C^n(\ph)=\mathcal{A}^n \oplus \mathcal{B}^{n-1}\]
with the differential given by $d(a,b)=(da,\ph(a)-db)$.  This cohomology fits, as expected, into an exact sequence involving $H^*(\mathcal{A})$ and $H^*(\mathcal{B})$.  Given a coefficient vector space $V$, $H^*(C^*,V)$ is the cohomology of the cochain complex $\Hom(V,C^*)$.

Now we state the main lemma of obstruction theory, which states the conditions under which a map can be extended over an elementary extension.
\begin{prop}[10.4 and 10.5 in \cite{GrMo}] \label{extExist}
  Let $\mathcal{A}\langle V \rangle$ be an $n$-dimensional elementary extension of a DGA $\mathcal{A}$.  Suppose we have a diagram of DGAs
  \[\xymatrix{
    \mathcal{A} \ar[r]^f \ar@{^{(}->}[d] & \mathcal{B} \ar[d]^h \\
    \mathcal{A}\langle V \rangle \ar[r]^g & \mathcal{C}
  }\]
  with $g|_{\mathcal{A}} \simeq hf$ by a homotopy $H:\mathcal{A} \to \mathcal{C} \otimes \mathbb{R}\langle t,dt \rangle$.  Then the map $O:V \to \mathcal{B}^{n+1} \oplus \mathcal{C}^n$ given by
  \[O(v)=\left(f(dv), g(v)+{\textstyle\int_0^1 H(dv)}\right)\]
  defines an obstruction class $[O] \in H^{n+1}(h:\mathcal{B} \to \mathcal{C};V)$ to producing an extension $\tilde f:\mathcal{A}\langle V \rangle \to \mathcal{B}$ of $f$ with $g \simeq h \circ \tilde f$ via a homotopy $\tilde H$ extending $H$.
  
  Moreover, if $h$ is surjective, then we can choose $H$ to be a constant homotopy.
\end{prop}
When the obstruction vanishes, there are maps $(b,c): V \to \mathcal{B}^n \oplus \mathcal{C}^{n-1}$ such that $d(b,c)=O$, that is,
\begin{align*}
  db(v) &= f(dv) \\
  dc(v) &= h\circ b(v)-g(v)-{\textstyle\int_0^1 H(dv)}.
\end{align*}
Then for $v \in V$ we can set $\tilde f(v)=b(v)$ and
\begin{equation} \label{extHtpy}
  \tilde H(v)=g(v)+d(c(v) \otimes t)+{\textstyle\int_0^t H(dv)}.
\end{equation}
This gives a specific formula for the extension.

\subsection{Homotopy groups} \label{S:pi_n}

One important application of the obstruction theory above is to make explicit the automorphism $V_n \cong \Hom(\pi_n(Y),\mathbb R)$; this approach was pioneered by Sullivan \cite[\S11]{SulLong}.  There are several ways of doing this.

First, suppose that $h:\mathcal B \to \mathcal C$ is a minimal model $m_{S^n}:\mathcal M_{S^n} \to \Omega^*_\flat(S^n)$.  Given a Lipschitz map $f:S^n \to Y$, we would like to understand its rational homotopy class.  By repeatedly applying Proposition \ref{extExist} to $f^*m_Y$, since $m_X$ is a quasi-isomorphism, we obtain a map $\ph:\mathcal M_Y^* \to \mathcal M_{S^n}^*$.  In particular, this restricts to a homomorphism $V_n \to \mathbb{R}$ representing an element of $\pi_n(Y) \otimes \mathbb R$.  Since $\ph|_{\mathcal M_Y^*(n-1)}$ is the zero map, we can equivalently think of the resulting homotopy as a partial nullhomotopy of $f^*m_Y$ which is obstructed in degree $n$.

Alternatively, we can take $h:\mathcal B \to \mathcal C$ to be the restriction homomorphism $\Omega^*_\flat(D^{n+1}) \to \Omega^*_\flat(S^n)$.  Again applying Proposition \ref{extExist} to $f^*m_Y$, since the restriction homomorphism is surjective, we obtain a formal extension $\ph:\mathcal M_Y^*(n-1) \to \Omega^*_\flat(D^{n+1})$ of $f^*m_Y$.  At that point, we obtain an obstruction in $\Hom(V_n,H^{n+1}(D^{n+1},S^n;\mathbb R))$ to extending to $V_n$.  This again represents an element of $\pi_n(Y) \otimes \mathbb R$.

Both of these constructions are used in quantitative arguments further down.  Quantitative arguments of this type were first given by Gromov in \cite[Ch.~7]{GrMS}.

\subsection{The shadowing principle}

The quantitative obstruction theory in \cite{PCDF} is built upon a combination of the coisoperimetric Lemma \ref{coIP} and algebraic properties of DGAs.  Thus all of the results there are true, mutatis mutandis, after expanding the universe from smooth to flat forms.  In particular, given a homomorphism $\mathcal{M}_Y^* \to \Omega_\flat^*(X)$, one can produce a nearby genuine map $X \to Y$ whose Lipschitz constant depends on geometric properties of the homomorphism.

To state this precisely, we first introduce more definitions.  Let $X$ and $Y$ be finite simplicial complexes or compact Riemannian manifolds such that $Y$ is simply connected and has a minimal model $m_Y:\mathcal{M}_Y^*\to\Omega_\flat^*Y$.  Fix norms on the finite-dimensional vector spaces $V_k$ of degree $k$ indecomposables of $\mathcal{M}_Y^*$; then for homomorphisms $\ph:\mathcal{M}_Y^* \to \Omega_\flat^*(X)$ we define the formal dilatation
\[\Dil(\ph)=\max_{2 \leq k \leq \dim X} \lVert\ph|_{V_k}\rVert_{\mathrm{op}}^{1/k},\]
where we use the $L^\infty$ norm on $\Omega_\flat^*(X)$.  Notice that if $f:X \to Y$ is an $L$-Lipschitz map, then $\Dil(f^*m_Y) \leq CL$, where the exact constant depends on the dimension of $X$, the minimal model on $Y$, and the norms.  Thus the dilatation is an algebraic analogue of the Lipschitz constant.

Given a formal homotopy
\[\Phi:\mathcal{M}_Y^* \to \Omega_\flat^*(X \times [0,T]),\]
we can define the dilatation $\Dil_T(\Phi)$ in a similar way.  The subscript indicates that we can always rescale $\Phi$ to spread over a smaller or larger interval, changing the dilatation; this is a formal analogue of defining separate Lipschitz constants in the time and space direction, although in the DGA world they are not so easily separable.

Now we can state some results from \cite{PCDF}.
\begin{thm}[{A special case of the shadowing principle, \cite[Thm.~4--1]{PCDF}}] \label{shadow}
  Let $\ph:\mathcal{M}_Y^* \to \Omega_\flat^*(X)$ be a homomorphism with $\Dil(\ph) \leq L$ which is formally homotopic to $f^*m_Y$ for some $f:X \to Y$.  Then $f$ is homotopic to a $g:X \to Y$ which is $C(X,Y)(L+1)$-Lipschitz and such that $g^*m_Y$ is homotopic to $\ph$ via a homotopy $\Phi$ with $\Dil_{1/L}(\Phi) \leq C(X,Y)(L+1)$.
\end{thm}
In other words, one can produce a genuine map by a small formal deformation of $\ph$.  We also present one relative version of this result:
\begin{thm}[{Cf.~\cite[Thm.~5--7]{PCDF}}] \label{relshadow}
  Let $f,g:X \to Y$ be two nullhomotopic $L$-Lipschitz maps and suppose that $f^*m_Y$ and $g^*m_Y$ are formally homotopic via a homotopy $\Phi:\mathcal{M}_Y^* \to \Omega_\flat^*(X \times [0,T])$ with $\Dil_T(\Phi) \leq L$.  Then there is a $C(X,Y)(L+1)$-Lipschitz homotopy $F:X \times [0,T] \to Y$ between $f$ and $g$.
\end{thm}
It is important for this result that the maps be nullhomotopic, rather than just in the same homotopy class.  This is because we did not require our formal homotopy to be in the relative homotopy class of a genuine homotopy.  In the zero homotopy class, one can always remedy this by a small modification, but in general the minimal size of the modification may depend in an opaque way on the homotopy class.

\subsection{The depth filtration}

Any minimal DGA has a filtration
\[0 \subseteq U_0 \subseteq U_1 \subseteq U_2 \subseteq \cdots\]
defined inductively as follows:
\begin{itemize}
\item $U_0$ is generated by all indecomposables with zero differential.
\item The product respects the filtration: if $u_1 \in U_i$ and $u_2 \in U_j$, then $u_1u_2 \in U_{i+j}$.
\item $U_i$ contains all indecomposables whose differentials are in $U_{i-1}$.
\end{itemize}
This filtration is canonical once one fixes the vector spaces of indecomposables.  We say that an element has \emph{depth} $i$ if it is contained in $U_i \setminus U_{i-1}$.

The filtration $\{U_i\}$ also induces a dual filtration
\[\pi_*(Y) \otimes \mathbb{F}=\Lambda_0 \supseteq \Lambda_1 \supseteq \Lambda_2 \supseteq \cdots\]
via the pairing between indecomposables in degree $n$ and $\pi_n(Y)$: $\alpha \in \Lambda_i \cap (\pi_n(Y) \otimes \mathbb F)$ if for every $j<i$, $U_n \cap V_j$ pairs trivially with $\alpha$.  In particular, $\Lambda_1$ is the kernel of the Hurewicz map with coefficients in $\mathbb F$.  This leads to a neat formulation of Gromov's distortion conjecture as discussed in the introduction:
\begin{defn}
  We say that $\pi_n(Y)$ \emph{satisfies Gromov's distortion conjecture} if any element of $\pi_n(Y) \cap (\Lambda_k \setminus \Lambda_{k+1})$ has distortion $\Theta(L^{n+k})$.
\end{defn}

\noindent The Leibniz rule tells us that if $x \in U_i$, then $dx \in U_{i-1}$.  By induction on $i$ one readily sees:
\begin{prop} \label{UtoU}
  Every DGA homomorphism respects the depth filtration, that is it sends $U_i$ into $U_i$.  Consequently, for every map between simply connected spaces, the induced homomorphism on rational homotopy groups respects the filtration $\{\Lambda_i\}$.
\end{prop}

The depth filtration allows us to define an alternate notion of ``size'' for homomorphisms $\ph:\mathcal{M}_Y^* \to \Omega^*X$, where $Y$ is compact and simply connected and $X$ is any metric complex.  Like dilatation above, this notion depends on the choice of norms on the spaces $V_k$, but this affects it only up to a constant depending on the dimension of $X$, since each $V_k$ is finite-dimensional.  Specifically, we define the \emph{$U$-dilatation} of $\ph$ to be
\[\Dil^U(\ph):=\max_{\substack{2 \leq n \leq \dim X\\0 \leq k<n}}
\lVert \ph|_{V_n \cap U_k} \rVert_{\mathrm{op}}^{\frac{1}{n+k}};\]
this is bounded above by dilatation but has significant advantages in that bounds on $U$-dilatation are often easy to preserve via inductive obstruction-theoretic constructions.

Using this formalism, we can show that the upper bound in Gromov's distortion conjecture holds in all cases:
\begin{prop} \label{dist<}
  For any simply connected $Y$, every element of $\pi_n(Y) \cap (\Lambda_k \setminus \Lambda_{k+1})$ has distortion $O(L^{n+k})$.
\end{prop}
Thus Gromov's distortion conjecture is satisfied if this bound is sharp.  We will see that this is the case for scalable spaces.
  

In effect, this estimate is obtained by keeping track of the size of the formal nullhomotopy produced in \S\ref{S:pi_n}.  We prove it as a corollary of a useful, albeit technical, proposition for all maps between simply connected spaces:
\begin{prop} \label{Udil-for-M}
  Let $X$ and $Y$ be two simply connected finite complexes with minimal models $m_X:\mathcal M_X^* \to \Omega_\flat^*(X)$ and $m_Y:\mathcal M_Y^* \to \Omega_\flat^*(Y)$.  Then for every $L$-Lipschitz map $f:X \to Y$, there is a map $\ph:\mathcal M_Y^* \to \mathcal M_X^*$ such that $f^*m_Y \simeq m_X\ph$ and
  \[\Dil^U(m_X\ph) \leq C(X,Y)L.\]
\end{prop}
\begin{proof}[Proof of Prop.~\ref{dist<}]
  Let $f:S^n \to Y$ be an $L$-Lipschitz map.  Then, as described in \S\ref{S:pi_n}, the corresponding DGA homomorphism $\ph:\mathcal M_Y^* \to \mathcal M_{S^n}^*$ constructed in Proposition \ref{Udil-for-M} restricts to a homomorphism $\ph_0:V_n \to \mathbb{R}$ which represents the pairing between $V_n$ and $[f]$.  In particular, if $f$ is a representative of $N\alpha \in \pi_n(Y) \cap (\Lambda_k \setminus \Lambda_{k+1})$, then there is some element $x \in V_n \cap U_k$ such that $\ph_0(x)=N$.  Therefore
  \[N \leq \Dil^U(\ph)^{n+k}\lvert x \rvert \leq C(\alpha)L^{n+k}.\]
  In other words, the distortion of $\alpha$ is $O(L^{n+k})$.
\end{proof}
\begin{proof}[Proof of Prop.~\ref{Udil-for-M}]
  Since $m_X$ is a quasi-isomorphism, there is a map $\ph$ which makes the diagram
  \[\xymatrix{
    \mathcal M_Y^* \ar@{-->}[r]^\ph \ar[d]_{m_Y} & \mathcal M_X^* \ar[d]^{m_X} \\
    \Omega_\flat^*(Y) \ar[r]^{f^*} & \Omega_\flat^*(X)
  }\]
  commute up to homotopy; we would like to build such a $\ph$ with controlled $U$-dilatation.  In the process, we will also need to control $\Dil^U_1(\Phi)$, where
  \[\Phi:\mathcal M_Y^* \to \Omega_\flat^*(X) \times \mathbb{R}\langle t,dt \rangle\]
  is the formal homotopy between $f^*m_Y$ and $m_X\ph$.
  
  We start by choosing a linear map $d^{-1}:d\mathcal M_X^* \to \mathcal M_X^*$ which gives a choice of primitive for each coboundary in $\mathcal M_X^*$.   We also choose a linear map $r:H^*(X;\mathbb R) \to \mathcal M_X^*$ fixing representatives for every cohomology class.  Since in every degree both the cohomology and the vector space of coboundaries are finite-dimensional, there are constants $C_n$, depending on the choices, such that for every $b \in d\mathcal M_X^n$,
  \[\lVert m_Xd^{-1}b \rVert_\infty \leq C_n\lVert m_Xb \rVert_\infty,\]
  and for every $a \in H^n(X;\mathbb Q)$,
  \[\lVert m_Xr(a) \rVert_\infty \leq C_n\inf \{\lVert \omega \rVert_\infty: \omega \in \Omega^n_\flat(X)\text{ and }[\omega]=a\}.\]
  
  We construct $\ph$ and $\Phi$ by induction on the stages of $\mathcal M_Y^*$, using Proposition \ref{extExist}.  Suppose we have built $\ph_n:\mathcal M_Y^*(n) \to \mathcal M_X^*$ and a homotopy
  \[\Phi_n:\mathcal M_Y^*(n) \to \Omega_\flat^*(X) \times \mathbb{R}\langle t,dt \rangle\]
  between $f^*m_Y|_{M_Y^*(n)}$ and $m_X\ph_n$ with the right estimates on dilatation.  We would like to extend both to $V_{n+1}$, the vector space of $(n+1)$-dimensional indecomposables of $Y$.  Fix a basis $\{v_i\}$ for $V_{n+1}$ which respects the depth filtration, that is, each $U_\ell \cap V_{n+1}$ is spanned by a subbasis.  For each $v_i \in U_\ell \setminus U_{\ell-1}$,
  \[\zeta(v_i)=m_Xd^{-1}\ph_n(dv_i)-f^*m_Y(v_i)-\textstyle{\int_0^1} \Phi_n(dv_i)\]
  is a closed form in $\Omega_\flat^{n+1}(X)$.  Its cohomology class has a representative $a(v_i) \in H^{n+1}(X;\mathbb R)$.  By Proposition \ref{coIP}, we can choose a primitive $\sigma(v_i)$ for $\zeta(v_i)-m_Xr(a(v_i))$ with
  \[\lVert \sigma(v_i) \rVert_\infty \leq C(n,Y)(1+C_n)\lVert \zeta(v_i) \rVert_\infty.\]
  In turn,
  \[\lVert \zeta(v_i) \rVert_\infty \leq C_n\Dil^U(\ph_n)^{(n+2)+(\ell-1)}+C(Y)L^{n+1}+\Dil^U_1(\Phi)^{(n+2)+(\ell-1)}.\]
  Then, using \eqref{extHtpy}, we choose
  \begin{align*}
    \ph_{n+1}(v_i) &= m_Xd^{-1}\ph_n(dv_i)-m_Xr(a(v_i)) \\
    \Phi_{n+1}(v_i) &= f^*m_Y(v_i)+d(\sigma(v_i) \otimes t)+\textstyle{\int_0^t} \Phi_n(dv_i).
  \end{align*}
  As desired, $\lVert \ph_{n+1}(v_i) \rVert_\infty$ and $\lVert \Phi_{n+1}(v_i) \rVert_\infty$ are both bounded by $C(X,Y)L^{n+\ell+1}$.
\end{proof}

Plugging the same estimates into the proof of \cite[Prop.~3--9]{PCDF} yields an additional technical lemma:
\begin{prop} \label{qLift}
  Suppose that $\Phi_k:\mathcal{M}_Y^*(k) \to \Omega^*_\flat(X) \otimes \mathbb{R}\langle t, dt \rangle$ is a partially defined homotopy between $\ph,\psi:\mathcal{M}_Y^* \to \Omega^*_\flat(X)$, and suppose that $\Dil^U(\ph)$, $\Dil^U(\psi)$, and $\Dil^U_1(\Phi_k)$ are all bounded by $L>0$.
  \begin{enumerate}[(i)]
  \item The obstruction to extending $\Phi_k$ to a homotopy
    \[\Phi_{k+1}:\mathcal{M}_Y^*(k+1) \to \Omega^*X \otimes \mathbb{R}\langle t,dt \rangle\]
    is a class in $H^k(X;V_{k+1})$ represented by a cochain whose restriction to $V_{k+1} \cap U_i$ has operator norm bounded by $C(k+1,Y)L^{k+1+i}$.
  \item If this obstruction class vanishes, then we can choose $\Phi_{k+1}$ so that $\Dil^U_1(\Phi_{k+1}) \leq C(k,Y)L$.
  \end{enumerate}
\end{prop}

\subsection{Formal spaces} \label{S:formal}

Many of the spaces we will be discussing in this paper are \emph{formal} in the sense of Sullivan.  A space $Y$ is formal if $\Omega^*Y$ is quasi-isomorphic to the cohomology ring $H^*(Y;\mathbb{R})$, viewed as a DGA with zero differential.  In other words, there is a map $\mathcal{M}_Y^* \to H^*(Y;\mathbb{R})$ which is a quasi-isomorphism.  (By \cite[Thm.~12.1]{SulLong}, the definition using any other ground field $\mathbb{F} \supseteq \mathbb{Q}$ is equivalent.)  More generally, we say a DGA is \emph{formal} if it is quasi-isomorphic to its cohomology ring.

Another way of saying this from the point of view of minimal models is this.  Formal spaces are those whose cohomology is a quotient of $\bigwedge U_0$, where $U_0$ is the subspace of indecomposables in the minimal model which have zero differential.  In other words, a minimal DGA is \emph{non-}formal if and only if it has a cohomology class which is not a linear combination of cup products of elements of $U_0$.

It follows that, while many rational homotopy types may have the same cohomology ring, there is exactly one formal one, and its minimal DGA can be constructed ``formally'' from the cohomology ring: at stage $k$, one adds generators that kill the relative $(k+1)$st cohomology of the map $\mathcal{M}_Y^*(k-1)(\mathbb{F}) \to H^*(Y;\mathbb{F})$.  This is the genesis of the term.

In fact, Halperin and Stasheff showed \cite[\S3]{HaSt} that for a formal space, one can choose the vector space of indecomposable generators\footnote{While the minimal model is unique up to isomorphism, such an isomorphism need not preserve this.} so that the depth filtration $\{U_i\}$ can be refined, non-canonically, to a bigrading $\mathcal{M}_Y^*=\bigwedge_i W_i$, where
\[(U_i \cap \text{indecomposables})
=W_i \oplus (U_{i-1} \cap \text{indecomposables}).\]

Spaces known to be formal include the simply connected symmetric spaces \cite{SulLong} and K\"ahler manifolds \cite{DGMS}, but there are many other examples, some of which are given in Table \ref{table}.

An important alternate characterization of formal spaces is that they are those $Y$ for which the \emph{grading automorphisms} $\rho_t:H^*(Y;\mathbb{F}) \to H^*(Y;\mathbb{F})$ taking $w \mapsto t^{\deg w}w$ lift to automorphisms of the minimal model \cite[Thm.~12.7]{SulLong}.  This lift is homotopically nonunique (for example, maps $S^2 \vee S^3 \to S^2 \vee S^3$ are characterized not only by the degrees on $S^2$ and $S^3$ but also by the Hopf invariant of the restriction-projection $S^3 \to S^2$) but all such lifts share certain properties.  In particular, all of them send $U_i$ to itself; moreover, given $w \in W_i \cap \mathcal{M}_Y^j$, they send $w \mapsto t^{i+j}w+w'$ where $w' \in U_{i-1}$.

Given a choice of $W_i$, one choice of lift sends every $w \in W_i \cap \mathcal{M}_Y^j$ to $t^{i+j}w$.  We refer to this as the automorphism associated to the bigrading $\{W_i\}$.

Similarly, after fixing a quasi-isomorphism $h_Y:\mathcal{M}_Y^*(\mathbb{Q}) \to H^*(Y;\mathbb{Q})$, the composition $\rho_t h_Y$ lifts to a canonical choice of automorphism of the minimal model, giving a ``one-parameter family'' of such automorphisms.  It turns out that we can always find enough genuine maps $Y \to Y$ implementing this choice:
\begin{thm}[{\cite[Corollary 1.1]{PWSM}}] \label{lem:Shiga}
  Let $Y$ be a formal finite CW complex.  There is an integer $t_0 \geq 1$, such that for every $z \in \mathbb{Z}$, $\rho_{zt_0}h_Y$ is realized by a genuine map $Y \to Y$.
\end{thm}
A result of this type was originally stated in \cite{Shiga}; see \cite{PWSM} for the proof as well as the full history.

\section{A non-formal example} \label{S:NF}
In this section we discuss an example space $Y$ which is not formal, but satisfies condition (iv) of Theorem \ref{tfae}: for $n<\dim Y$, nullhomotopic $L$-Lipschitz maps $S^n \to Y$ have $O(L)$-Lipschitz nullhomotopies.  On the other hand, nullhomotopies of maps from higher-dimensional spheres cannot be made linear (so condition (iii) is not satisfied).  This demonstrates that the method of proof of Theorem \ref{tfae}, which relies on induction by skeleta to show that (iv) implies the other conditions, cannot be straightforwardly extended to show that non-formal spaces never admit linear nullhomotopies.  On the other hand, we also do not have a candidate non-formal space which could admit linear nullhomotopies from all domains.  Thus the following question remains open:
\begin{question}
  Do non-formal simply connected targets ever admit linear nullhomotopies of maps from all compact domains?  For that matter, from all spheres?
\end{question}
Our space is 8-dimensional, although a 6-dimensional example can also be constructed.  Namely, we take the CW complex
\[Y=(S^3_a \vee S^3_b \vee S^5) \cup_f e^8,\]
where $f:S^7 \to S^3 \vee S^3 \vee S^5$ is given by the iterated Whitehead product
\[\bigl[\id_a,\id_{S^5}+[\id_a,\id_b]\bigr],\]
with $\id_a$ and $\id_b$ representing the identity maps on the two copies of $S^3$.

Above and below we use the following conventions to define representatives of homotopy classes with good Lipschitz constants.  Let $\ph:S^k \to Y$ and $\psi:S^\ell \to Y$ be maps with Lipschitz constant $\leq L$.  The notation $[\ph,\psi]$ represents the standard Whitehead product of $\ph$ and $\psi$, that is the $C(k,\ell)L$-Lipschitz map $S^{k+\ell-1} \to Y$ given by composing $\ph \vee \psi$ with the attaching map of the $(k+\ell)$-cell of $S^k \times S^\ell$.  The notation $N\ph$ represents the composition of $\ph$ with a degree $N$, $O(N^{1/k})$-Lipschitz map $S^k \to S^k$, as constructed in Proposition \ref{SntoSn}.  Finally, if $k=\ell$, then $\ph+\psi$ represents the $C(k)L$-Lipschitz map given by composing $\ph \vee \psi$ with a map sending the northern and southern hemisphere to different copies of the sphere.
\begin{prop} \label{NF:yes}
  For $n \leq 7$, nullhomotopic maps $S^n \to Y$ have linear nullhomotopies.
\end{prop}
\begin{prop} \label{NF:no}
  There is a sequence of nullhomotopic maps $g_N:S^{13} \to Y$ with Lipschitz constant $O(N)$ but such that every nullhomotopy of $g_N$ has Lipschitz constant $\Omega(N^{17/16})$.
\end{prop}
\begin{proof}[Proof of Prop.~\ref{NF:yes}.]
  For $n \leq 7$, any $L$-Lipschitz map $S^n \to Y$ has an $O(L)$-Lipschitz homotopy to one whose image lies in the 7-skeleton of $Y$, $W=S^3 \vee S^3 \vee S^5$.  Moreover, if $n<7$, such a map is nullhomotopic in $Y$ if and only if it is nullhomotopic in $W$.  Since this is a scalable space, any such nullhomotopy can be made $O(L)$-Lipschitz by Theorem \ref{tfae}.

  There remains the case $n=7$.  Clearly a map $g:S^7 \to W$ is nullhomotopic in $Y$ if and only if it is in the homotopy class $N[f] \in \pi_7(W)$ for some $N$.  By Theorem \ref{props}\eqref{B:dist}, the distortion of $[f]$ in $W$ is $\sim L^8$, meaning that if $g$ is $L$-Lipschitz, it is homotopic in $W$ to the $O(L)$-Lipschitz map
  \[g'=\bigl[A\id_{S^3_a},B\bigl(\id_{S^5}+[\id_a,\id_b]\bigr)\bigr]+Cf\]
  for $A \lesssim L^3$, $B \lesssim L^5$, and $C \lesssim L^7$, and again by Theorem \ref{tfae}, this homotopy can be made $O(L)$-Lipschitz.

  Finally, we need to show that $g'$ has an $O(L)$-Lipschitz nullhomotopy in $Y$.  So consider a map $p:S^3 \times S^5 \to Y$ sending the $S^3$ factor to $S^3_a$ and the $S^5$ factor to $Y$ via $\id_{S^5}+[\id_a,\id_b]$.  Since $S^3 \times S^5$ is scalable, the map
  \[[A\id_{S^3},B\id_{S^5}]+C[\id_{S^3},\id_{S^5}]\]
  for $A \lesssim L^3$, $B \lesssim L^5$, and $C \lesssim L^7$ has an $O(L)$-Lipschitz nullhomotopy there.  Pushing this nullhomotopy to $Y$ via $p$ gives an $O(L)$-Lipschitz nullhomotopy of $g'$.
\end{proof}
\begin{proof}[Proof of Prop.~\ref{NF:no}.]
  The map
  \[g_N=[[N^3\id_a,N^5\id_{S^5}],[N^3\id_a,[N^3\id_a,N^3\id_b]]]\]
  is $O(N)$-Lipschitz, and it is homotopic in $S^3_a \vee S^3_b \vee S^5$ to $[[N^3\id_a,N^5\id_{S^5}],N^9f]$ and therefore nullhomotopic in $Y$.  We will show that any nullhomotopy has Lipschitz constant $\Omega(N^{17/16})$.

  We will need to understand some of the rational homotopy theory of the subspace
  \[W=S^3_a \vee S^3_b \vee S^5 \subset Y.\]
  We note that $W$ is formal and therefore its minimal DGA can be computed formally. Here are some of the generators in low dimensions (the number $n$ in $x^{(n)}$ denotes the degree of a generator $x$):
  \[\mathcal{M}_W^* \supset \left\langle \begin{array}{r | l}
    a^{(3)},b^{(3)},c^{(5)},u_b^{(5)} & da=db=dc=0, du_b=ab \\
      u_c^{(7)}, v_b^{(7)}, w_b^{(9)}, v_c^{(9)} &
      du_c=ac, dv_b=au_b, dw_b=av_b, dv_c=au_c \\
    w_c^{(11)}, z^{(13)} & dw_c=av_c, dz=u_cv_b-v_cu_b-cw_b+w_cb
  \end{array}\right\rangle.\]
  We will show two facts: first, $\langle z,[g_N] \rangle \sim N^{17}$; second, if $F:(D^{14},\partial D^{14}) \to (Y,W)$ is an $L$-Lipschitz map, then $\langle z, [F|_\partial] \rangle=O(L^{16})$.  Therefore, if $F$ is a nullhomotopy of $g_N$, then its Lipschitz constant is $\Omega(N^{17/16})$.

  Since $[g_N]=N^{17}[[\id_a,\id_{S^5}],[\id_a,[\id_a,\id_b]]]=N^{17}[g_1]$, to see that $\langle z,[g_N] \rangle \sim N^{17}$, it is enough to show that the pairing $\langle z, [g_1] \rangle$ is nontrivial.  As explained in \cite[\S 13(e)]{FHT}, the Whitehead product is dual to the quadratic part of the differential in the minimal model.  In particular, $u_b$ is dual to $[\id_a,\id_b]$ and $u_c$ is dual to $[\id_a,\id_{S^5}]$; therefore $v_b$ is dual to $[\id_a,[\id_a,\id_b]]$; and finally, since $dz$ contains the term $u_cv_b$, $z$ pairs nontrivially with $[g_1]$.

  Now suppose that $F:(D^{14},\partial D^{14}) \to (Y,W)$ is an $L$-Lipschitz map.  We can compute the pairing $\langle z, [F|_\partial] \rangle$ using the second method discussed in \S\ref{S:pi_n}.  Fix a minimal model $m_W:(\mathcal{M}_W^*,d)\to\Omega^*W$; we attempt to extend $(F|_\partial)^*m_W$ to a map $\epsi:\mathcal{M}_W^* \to \Omega^*D^{14}$.  Since the relative cohomology is zero through dimension 13, we do not encounter an obstruction until we try to extend to 13-dimensional indecomposables.  At that point, regardless of previous choices, the obstruction to extending to $z$ is given by the pairing, that is,
  \[\int_{D^{14}} \epsi(dz)=\langle z, [F|_\partial] \rangle.\]
  We will use the map $F$ to build one such extension with bounds on the sizes of the forms; in particular we will make sure that $\lVert \epsi(dz) \rVert=O(L^{16})$, so that the pairing is also $O(L^{16})$.

  By \cite[\S 13(d)]{FHT}, $m_W$ can be extended to a quasi-isomorphism $m_Y:(\mathcal{M}_W^* \oplus \mathbb{R}y,d') \to \Omega^*Y$ which we use as a non-minimal model for $Y$.  Here $y$ satisfies $y^2=0=dy$ and $xy=0$ for every $x \in \mathcal{M}_W^*$, and $d'=d$ except for 7-dimensional indecomposables $x$ in $\mathcal{M}_W^*$, for which
  \[d'x=dx+\langle x,[f] \rangle y.\]
  In particular, $m_Yy$ is a closed form concentrated in the interior of the 8-cell, representing the fundamental class of $H^8(Y,W;\mathbb{R})$.

  To build the extension of $(F_\partial)^*m_W$ to $\Omega^*D^{14}$, we first send
  \[a \mapsto F^*m_Ya,\quad b \mapsto F^*m_Yb,\quad c \mapsto F^*m_Yc,\quad
  u_b \mapsto F^*m_Yu_b;\]
  then choose a 7-form $\omega \in \Omega^*(D^{14},\partial D^{14})$ satisfying
  $d\omega=F^*m_Yy$ and $\lVert\omega\rVert_\infty=O(L^8)$ and send
  \[u_c \mapsto F^*m_Yu_c-\langle u_c,[f] \rangle\omega, \quad
  v_b \mapsto F^*m_Yv_b-\langle v_b,[f] \rangle\omega;\]
  and finally, using Lemma \ref{coIP}, pick forms $\epsi(w_b)$, $\epsi(v_c)$, and $\epsi(w_c)$ satisfying
  \[\lVert\epsi(w_b)\rVert_\infty=O(L^{11}), \quad
  \lVert\epsi(v_c)\rVert_\infty=O(L^{11}), \quad
  \lVert\epsi(w_c)\rVert_\infty=O(L^{13}).\]
  This construction gives us $\lVert \epsi(dz) \rVert_\infty=O(L^{16})$.
\end{proof}

\section{Proof of Theorem \ref{tfae}} \label{S5}
In this section we prove Theorem \ref{tfae} together with Theorem \ref{props}(a).
First, we restate these results:
\begin{thm*}
  The following are equivalent for a simply connected finite complex $Y$:
  \begin{enumerate}[(i)]
  \item There is a DGA homomorphism $i:H^*(Y;\mathbb{R}) \to \Omega_\flat^*Y$ which sends each cohomology class to a representative of that class.
  \item There is a constant $C(Y)$ and infinitely many (indeed, a logarithmically dense set of) $p \in \mathbb{N}$ such that there is a $C(Y)(p+1)$-Lipschitz self-map which induces multiplication by $p^n$ on $H^n(Y;\mathbb{R})$.
  \item $Y$ is formal, and for all finite simplicial complexes $X$, nullhomotopic $L$-Lipschitz maps $X \to Y$ have $C(X,Y)(L+1)$-Lipschitz nullhomotopies.
  \item $Y$ is formal, and for all $n<\dim Y$, nullhomotopic $L$-Lipschitz maps $S^n \to Y$ have $C(X,Y)(L+1)$-Lipschitz homotopies.
  \end{enumerate}
  Moreover, this property is a rational homotopy invariant.
\end{thm*}
\begin{proof}
  We start by proving the equivalence of (i) and (ii), followed by rational invariance; the statements on homotopies are the most involved and are deferred to the end.

  \subsubsection*{(i) $\Rightarrow$ (ii)}
  We start by showing:
  \begin{lem}
    A space satisfying (i) is formal.
  \end{lem}
  \begin{proof}
    The homomorphism $i:H^*(Y; \mathbb R) \to \Omega_\flat^*Y$ guaranteed by (i) is a quasi-isomorphism.  Let $m_Y':\mathcal M_Y^* \to \Omega_\flat^*Y$ be a minimal model.  Then by repeated applications of Proposition \ref{extExist}, we get a quasi-isomorphism $h_Y:\mathcal M_Y^* \to H^*(Y;\mathbb R)$ such that $i \circ h_Y \simeq m_Y$.  This shows that $h_Y$ is a quasi-isomorphism of DGAs over the reals, in other words $\mathcal M_Y^*$ is formal.
  \end{proof}
  The map $i \circ q$ constructed in the lemma is a minimal model $m_Y=i \circ h_Y:\mathcal{M}^*_Y \to \Omega^*Y$ which sends all homologically trivial elements to 0.

  Let $\rho_t:H^*(Y) \to H^*(Y)$ be the grading automorphism which multiplies $H^k$ by $t^k$.  Then by Theorem \ref{lem:Shiga}, there is some $t>1$ such that $\rho_t \circ h_Y$ is realized by a genuine map $f:Y \to Y$.  For every $N$, $f^N$ is in the rational homotopy class of the map $i\rho_{t^N}h_Y:\mathcal{M}_Y^* \to \Omega^*Y$, whose dilatation is $O(t)$; therefore, by the shadowing principle \ref{shadow}, we can build an $O(t^N)$-Lipschitz map in this homotopy class.  Therefore such maps are at least logarithmically dense.

  \subsubsection*{(ii) $\Rightarrow$ (i)}
  Suppose that there is an infinite sequence of $p \in \mathbb{N}$ and $C(Y)(p+1)$-Lipschitz maps $r_p$ as given.  Let $m_Y:\mathcal{M}_Y^* \to \Omega^*Y$ be a minimal model, and
  \[\mathcal{M}_Y^*=\bigwedge_{\ell=1}^\infty W_\ell\]
  a bigrading as described in \S\ref{S:formal}.  There is an automorphism $r_p$ of $\mathcal{M}_Y^*$ extending the grading automorphism on $H^*(Y)$ which sends $a \in W_\ell$ to $p^{\deg a+\ell}a$.

  Now for each $p$ the map $\ph_p=r_p^*m_Y\rho_{1/p}:\mathcal{M}_Y^* \to \Omega_\flat^*(Y)$ sends
  \[w \mapsto \frac{1}{p^{\ell+\deg w}}r_p^*m_Yw, \qquad w \in W_\ell.\]
  This sequence of maps is uniformly bounded, and therefore has a subsequence
  which weak$^\flat$-converges to some $\ph_\infty$.
  \begin{lem}
    For indecomposables $w$, $\ph_\infty(w)=0$ if and only if
    $w \in \bigoplus_{\ell=1}^\infty W_\ell$.
  \end{lem}
  \begin{proof}
    If $w \in \bigoplus_{\ell=1}^\infty W_\ell$, then its image is zero since $\lVert r_p^*m_Yw \rVert_\infty \leq [C(Y)(p+1)]^{\deg w}$.  On the other hand, if $w \in W_0$, then it is cohomologically nontrivial, and thus there is a flat cycle $A$ and a $C_A>0$ such that $\int_A r_p^*m_Yw=C_Ap^{\deg w}$ for every $p$.  Thus $\int_A \ph_\infty(w)=C_A$ and so $\ph_\infty(w) \neq 0$.
  \end{proof}
  Now, if an element $w \in \bigwedge W_0$ is zero in $H^*(Y;\mathbb{R})$, then it is the differential of some element of $W_1$ and therefore again $\ph_\infty(w)=0$.  Thus $\ph_\infty:\mathcal{M}_Y^* \to \Omega_\flat^*(Y)$ factors through $H^*(Y;\mathbb{R})$, showing (i).

  \subsubsection*{Rational homotopy invariance of (ii)}
  Suppose that $Y$ has property (ii) and $Z$ is a rationally equivalent finite complex.  By \cite[Theorem B]{PWSM}, there are maps $Z \xrightarrow{f} Y \xrightarrow{g} Z$ inducing rational homotopy equivalences such that $g \circ f$ induces the automorphism $\rho_q$ for some $q$.  Then we can get a sequence of maps verifying (ii) for $Z$ by composing
  \[Z \xrightarrow{f} Y \xrightarrow{r_p} Y \xrightarrow{g} Z\]
  for each $p$ in the sequence verifying (ii) for $Y$.

  \subsubsection*{(iii) $\Rightarrow$ (iv)}
  This is clear.

  \subsubsection*{(iv) $\Rightarrow$ (ii)}
  Suppose that $Y$ is formal and admits linear nullhomotopies of maps from $S^n$.  Theorem \ref{lem:Shiga} gives a way of realizing the grading automorphism $\rho_t$ of $Y$ by a map $r_t:Y \to Y$ for some infinite, logarithmically dense sequence of $t$, but without geometric constraints.  It thus remains to construct homotopic maps with Lipschitz constant $O(L)$.  We defer the details to the next section as they require some additional technical machinery from \cite{PCDF}.

  In fact, our construction will give a more general result, which may be thought of as a strengthening of the shadowing principle for scalable spaces:
  \begin{lem} \label{strongShadows}
    Suppose $Y$ admits linear nullhomotopies of maps from $S^k$, $k \leq n-1$.  Let $X$ be an $n$-dimensional simplicial complex, and let $\ph:\mathcal{M}^*_Y \to \Omega^*_\flat(X)$ be a homomorphism which satisfies
    \[\Dil^U(\ph) \leq L,\]
    and which is formally homotopic to $f^*m_Y$ for some $f:X \to Y$.  Then there is a $g:X \to Y$ which is $C(n,Y)(L+1)$-Lipschitz and homotopic to $f$, where $C(n,Y)$ depends on the choices of norms on $V_k$.
  \end{lem}
  As a special case, in combination with Theorem \ref{lem:Shiga}, we see that such a $Y$ satisfies (ii).  Formally, this lemma also implies Gromov's distortion conjecture for $Y$, Theorem \ref{props}(\ref{B:dist}).  In fact, though, we will prove this separately and use it in the proof.

  \subsubsection*{(ii) $\Rightarrow$ (iii)}
  Let $X$ be a finite simplicial complex and $f:X \to Y$ a nullhomotopic $L$-Lipschitz map.  Choose a natural number $t>1$ such that there is an map $r_t:Y \to Y$ which induces the grading automorphism $\rho_t$ on cohomology.

  We will define a nullhomotopy of $f$ by homotoping through a series of maps which are more and more ``locally organized''.  Specifically, for $1 \leq k \leq s=\lceil\log_pL\rceil$, we build a $C(X,Y)(L/p^k+1)$-Lipschitz map $f_k:X \to Y$ by applying the shadowing principle \ref{shadow} to the map
  \[f^*m_Y\rho_{p^{-k}}:\mathcal{M}_Y^* \to \Omega^*X.\]
  We will build a nullhomotopy of $f$ through the sequence of maps
  \[\xymatrix{
    f \ar@{-}[r] & r_p \circ f_1 \ar@{-}[rd] & r_{p^2} \circ f_2 \ar@{-}[rd] &
    \ldots \ar@{-}[rd] & r_{p^s} \circ f_s \ar@{-}[r] & \text{const}. \\
    && r_p \circ r_p \circ f_2 \ar@{-}[u] & \ldots &
    r_{p^{s-1}} \circ r_p \circ f_s \ar@{-}[u]
  }\]
  As we go right, the \emph{length} (Lipschitz constant in the time direction) of the $k$th intermediate homotopy increases---it is $O(p^k)$---while the \emph{thickness} (Lipschitz constant in the space direction) remains $O(L)$. Thus all together, these homotopies can be glued into an $O(L)$-Lipschitz nullhomotopy of $f$.

  Informally, the intermediate maps $r_{p^k} \circ f_k$ look at scale $p^k/L$ like thickness-$p^k$ ``bundles'' or ``cables'' of identical standard maps at scale $1/L$.  This structure makes them essentially as easy to nullhomotope as $L/p^k$-Lipschitz maps.

  We now build the aforementioned homotopies:
  \begin{lem} \label{lem:rp}
    There is an $O(p^k)$-Lipschitz homotopy $F_k:Y \times [0,1] \to Y$ between $r_{p^k}$ and $r_{p^{k-1}} \circ r_p$.
  \end{lem}
  \begin{lem} \label{lem:fk}
    There is a thickness-$O(L/p^k)$, constant length homotopy $G_k:X \times [0,1] \to Y$ between $f_k$ and $r_p \circ f_{k+1}$.
  \end{lem}
  This induces homotopies of thickness $O(L)$ and length $O(p^k)$:
  \begin{itemize}
  \item $F_k \circ (f_k \times \id)$ from $r_{p^{k-1}} \circ r_p \circ f_k$ to $r_{p^k} \circ f_k$;
  \item $r_{p^k} \circ G_k$ from $r_{p^k} \circ f_k$ to $r_{p^k} \circ r_p \circ f_{k+1}$.
  \end{itemize}
  Finally, the map $f_s$ is $C(X,Y)$-Lipschitz and therefore has a short homotopy to one of a finite set of nullhomotopic simplicial maps $X \to Y$.  For each map in this finite set, we can pick a fixed nullhomotopy, giving a constant bound for the Lipschitz constant of a nullhomotopy of $f_s$ and therefore a linear one for $r_{p^s} \circ f_s$.

  Adding up the lengths of all these homotopies gives a geometric series which sums to $O(L)$, completing the proof of the theorem modulo the two lemmas above.
\end{proof}
\begin{proof}[Proof of Lemma \ref{lem:rp}.]
  We use the fact that the maps $r_{p^i}$ were built using the shadowing principle.  Thus, there are formal homotopies $\Phi_i$ of length $C(X,Y)$ between $m_Y\rho_{p^i}$ and $r_{p^i}^*m_Y$.  This allows us to construct the following formal homotopies:
  \begin{itemize}
  \item $\Phi_k$, time-reversed, between $r_{p^k}^*m_Y$ and $m_Y\rho_{p^k}$, of length $C(X,Y)$;
  \item $\Phi_1\rho_{p^{k-1}}$ between $m_Y\rho_{p^k}$ and $r_p^*m_Y\rho_{p^{k-1}}$, of length $C(X,Y)p^{k-1}$;
  \item and $(r_{p^{k-1}}^* \otimes \id)\Phi_{k-1}$ between $r_p^*m_Y\rho_{p^{k-1}}$ and $r_p^*r_{p^{k-1}}^*m_Y$, of length $C(X,Y)$.
  \end{itemize}
  Concatenating these three homotopies and applying the relative shadowing principle \ref{relshadow} to the resulting map $\mathcal{M}^*_Y \to \Omega^*(Y \times [0,1])$ rel ends, we get a linear thickness homotopy of length $O(p^{k-1})$ between the two maps.
\end{proof}
\begin{proof}[Proof of Lemma \ref{lem:fk}.]
  We use the fact that the maps $f_k$ and $f_{k+1}$ were built using the shadowing principle.  Thus there are formal homotopies $\Psi_i$ of length $C(X,Y)$ between $f^*m_Y\rho_{p^{-i}}$ and $f_i$.  This allows us to construct the following formal homotopies:
  \begin{itemize}
  \item $\Psi_k$, time-reversed, between $f_k$ and $f^*m_Y\rho_{p^{-k}}$, of length $C(X,Y)$;
  \item $\Psi_{k+1}\rho_p$ between $f^*m_Y\rho_{p^{-k}}$ and $f_{k+1}^*m_Y\rho_p$, of length $C(X,Y)p$;
  \item and $(f_{k+1}^* \otimes \id)\Phi_1$ between $f_{k+1}^*m_Y\rho_p$ and $r_p^*f_{k+1}^*m_Y$, of length $C(X,Y)$.
  \end{itemize}
  Concatenating these three homotopies and applying the relative shadowing principle \ref{relshadow} to the resulting map $\mathcal{M}^*_Y \to \Omega^*(X \times [0,1])$ rel ends, we get a linear thickness homotopy of length $O(p)$ between the two maps.
\end{proof}

\section{Proof of Theorem \ref{props}} \label{S:B}
Now we prove Theorem \ref{props}, which we again restate:
\begin{thm*}[Properties of scalable spaces] \ 
  \begin{enumerate}[(a)]
  \item Scalability is invariant under rational homotopy equivalence.
  \item The class of scalable spaces is closed under products and wedge sums.
  \item All $n$-skeleta of scalable complexes are scalable, for $n \geq 2$.
  \item Scalable spaces satisfy Gromov's distortion conjecture.  That is, if $Y$ is scalable, all elements $\alpha \in \pi_k(Y) \cap \Lambda_\ell$ outside $\Lambda_{\ell+1}$ have distortion $\Theta(L^{k+\ell})$.
  \end{enumerate}
\end{thm*}
In fact, we will show this for spaces that satisfy (i) and (ii) of Theorem \ref{tfae}.  Thus it doesn't matter that we are not done proving that (iii) and (iv) are equivalent to (i) and (ii), and indeed we will use part (d) in the proof of Lemma \ref{strongShadows}.

We already showed (a) in the previous section.  For part (b), if (ii) holds for spaces $X$ and $Y$, then we can take the product and wedge sum of the respective scaling maps to get scaling maps of $X \times Y$ and $X \vee Y$.

To show part (c), we use (i).  Let $Y$ be a complex satisfying (i), and let $n \geq 2$.  The inclusion $i:Y^{(n)} \hookrightarrow Y$ induces isomorphisms for $H^k$ for $k<n$ and an injection on $H^n$.  Then the homomorphism $H^*(Y;\mathbb R) \to \Omega^*_\flat(Y)$ composed with the restriction to forms on $Y^{(n)}$ gives a homomorphism from $i^*H^*(Y;\mathbb R) \subseteq H^*(Y^{(n)}; \mathbb R)$ to $\Omega^*_\flat(Y^{(n)})$.  To extend to the rest of $H^n(Y^{(n)}; \mathbb R)$, we can choose any $n$-forms representing a basis for a complementary subspace and extend by linearity.  Since these forms are top-dimensional, their wedge product with any other form is zero, as desired.

Part (d) is a mild generalization of \cite[Theorem 5--4]{PCDF} showing that symmetric spaces satisfy Gromov's distortion conjecture, whose proof already uses the fact that they satisfy (i).  Suppose first that $\alpha \in \pi_k(Y)$ is contained in $\Lambda_\ell$.  We will show that its distortion is $\Omega(L^{k+\ell})$.  Let $f:S^k \to Y$ be a representative of $\alpha$, and let $r_p$ be maps realizing (ii) for any $p$ for which they exist.  Then $r_pf$ is an $O(L)$-Lipschitz representative of $q^{k+\ell}\alpha$.  Such a map $r_p$ exists for at least a logarithmically dense set of integers $p$, so all other multiples can also be represented with a similar Lipschitz constant.
  
The other side of the inequality follows immediately from Proposition \ref{dist<}.

\section{Maps to scalable spaces}
The purpose of this section is to prove Lemma \ref{strongShadows}, which we restate here:
\begin{lem*}
  Suppose $Y$ admits linear nullhomotopies of maps from $S^k$, $k \leq n-1$.  Let $X$ be an $n$-dimensional simplicial complex, and let $\ph:\mathcal{M}^*_Y \to \Omega^*_\flat(X)$ be a homomorphism which satisfies
  \[\Dil^U(\ph) \leq L,\]
  and which is DGA homotopic to $f^*m_Y$ for some $f:X \to Y$.  Then there is a map $g:X \to Y$ which is $C(n,Y)(L+1)$-Lipschitz and homotopic to $f$, where the constant $C(n,Y)$ depends on the choices of norms on $V_k$.
\end{lem*}
Both scalability and Gromov's distortion conjecture follow as corollaries of this lemma.  These should be thought of as instances of a wider principle that the lemma facilitates the construction of maximally efficient maps.  While the original shadowing principle gives a close relationship between the (usual) dilatation of the ``most efficient'' homomorphism $\mathcal{M}_Y^* \to \Omega^*X$ and the best Lipschitz constant of a map $X \to Y$ in a given homotopy class, the homomorphisms involved can be as difficult to construct as the maps.  On the other hand, in light of Proposition \ref{Udil-for-M}, homomorphisms with optimal $U$-dilatation can always be constructed by factoring through maps between minimal models.  This means that for scalable spaces, the Lipschitz norm of a homotopy class can be computed by studying the maps between minimal models which represent it.  Although the set of such maps may be quite complicated in general, computing it is at least a finite obstruction-theoretic problem.  We summarize this as a theorem:
\begin{thm} \label{MM-only}
  Given two simply connected spaces $X$ and $Y$ and a minimal model $m_X:\mathcal M_X^* \to \Omega^*_\flat(X)$, we can define two norms on the set of homotopy classes $\alpha \in [X,Y]$:
  \begin{align*}
    \lvert \alpha \rvert_{\Lip} &= \inf \{\Lip(f) \mid f:X \to Y, [f]=\alpha\} \\
    \lvert \alpha \rvert_{\mathrm{UDil}} &= \inf \{\Dil^U(m_X^*\ph) \mid \ph:\mathcal M_Y^* \to \mathcal M_X^*, m_X^*\ph \simeq f^*m_Y\text{ where }[f]=\alpha\}.
  \end{align*}
  If $Y$ is a scalable space, then there are constants $c,C>0$ such that for every $\alpha \in [X,Y]$,
  \[c\lvert \alpha \rvert_{\mathrm{UDil}} \leq \lvert \alpha \rvert_{\Lip} \leq C(\lvert \alpha \rvert_{\mathrm{UDil}}+1).\]
\end{thm}
Note that this is true regardless of the choice of $m_X$, and therefore the norm $\lvert\alpha\rvert_{\mathrm{UDil}}$ essentially only depends on information about minimal models.

We prove Lemma \ref{strongShadows} by induction using the following statements:
\begin{enumerate}[($a_n$)]
\item Lemma \ref{strongShadows} holds through dimension $n$ (we make this more precise during the proof, but in particular it holds for $n$-dimensional $X$).
\item If $Z$ is an $n$-complex which is formal and admits linear nullhomotopies of maps from $S^k$, $k<n$, then it satisfies (ii).
\end{enumerate}
Clearly, ($a_n$) implies ($b_n$).  In particular, since any skeleton of a formal space is formal \cite[Lemma 3.1]{Shiga}, $Y^{(n)}$ satisfies (ii).  Therefore, according to Theorem \ref{props}(d), it also satisfies Gromov's distortion conjecture.  In addition, we need the following easy extension of that result:
\begin{lem} \label{ref-dist}
  Suppose that $Y$ is an $n$-complex and the distortion conjecture holds for $\pi_n(Y^{(n-1)})$.  Then it holds for $\pi_n(Y)$.
\end{lem}
\begin{proof}
  This follows from the exact sequence
  \[\cdots \to \pi_n(Y^{(n-1)}) \xrightarrow{i} \pi_n(Y) \xrightarrow{j} \pi_n(Y,Y^{(n-1)}) \to \cdots.\]
  Since $\img j \subseteq H_n(Y)$, all elements of $\pi_n(Y)$ not in $\ker j$ are undistorted.  Conversely, elements in the image of $i$ are at least as distorted as their preimages.  To show that this is consistent with the distortion conjecture, we must analyze the induced map $\mathcal{M}_Y^* \to \mathcal{M}_{Y^{(n-1)}}^*$.  In fact, this map is injective in degrees $\leq n-1$ (and hence preserves the filtration by the $U_j$) and all extra $n$-dimensional generators of $\mathcal{M}_Y^*$ have zero differential; see \cite[\S13(d)]{FHT}.  This completes the proof of the lemma.
\end{proof}
We now proceed with the proof of the inductive step.
\begin{proof}[Proof of Lemma \ref{strongShadows}.]
  The structure of this proof is very similar to the original proof of the shadowing principle in \cite[\S4]{PCDF}.  That is, we pull $f$ to a map with small Lipschitz constant skeleton by skeleton, all the while using $\ph$ as a model to ensure that we don't end up with overly large obstructions at the next stage (as might occur if we pulled in an arbitrary way.)  The biggest difference is that we don't need to subdivide before performing the induction.

  The details follow.  Suppose, as an inductive hypothesis, that we have constructed the following data:
  \begin{itemize}
  \item A map $g_k:X \to Y$, homotopic to $f$, whose restriction to $X^{(k)}$ is $C(k,Y)(L+1)$-Lipschitz.
  \item A homotopy
    $\Phi_k:\mathcal{M}_Y^* \to \Omega_\flat^*(X) \otimes \mathbb{R}(t,dt)$ from $g_k^*m_Y$ to $\ph$ such that
    \[\Dil_1^U((\Phi_k|_{\mathcal{M}_Y^*(k)})|_{X^{(k)}}) \leq C(k,Y)(L+1).\]
  \end{itemize}
  We write $\beta_k=\int_0^1 \Phi_k$; note that for $v \in V_i$,
  \[d\beta_k(v)=\ph(v)-g_k^*m_Y(v)-\textstyle{\int_0^1} \Phi_k(dv)\]
  and $\beta_k(v)|_A=0$.

  We then construct the analogous data one dimension higher, starting with $g_{k+1}$.  Let $b \in C^k(X;\pi_{k+1}(Y))$ be the simplicial cochain obtained by integrating $\beta_k|_{V_{k+1}}$ over $k$-simplices and choosing an element of $\pi_{k+1}(Y)$ whose image in $V_{k+1}$ is as close as possible in norm (but otherwise arbitrary.)  Note that the values of $b$ are not a priori bounded in any way.  We use $b$ to specify a homotopy $H_{k+1}:X \times [0,1] \to Y$ from $g_k$ to a new map $g_{k+1}$.

  We start by setting $H_{k+1}$ to be constant on $X^{(k-1)}$.  On each $k$-simplex $q$, we set $H_{k+1}|_q$ to be a map such that
  \[g_{k+1}|_q=H_{k+1}|_{q \times \{1\}}=H_{k+1}|_{q \times \{0\}}=g_k|_q,\]
  but such that on the cell $q \times [0,1]$, the map traces out the element $\langle b, q \rangle \in \pi_{k+1}(Y)$.  This is well-defined since $H_{k+1}|_{\partial(q \times [0,1])}$ is canonically nullhomotopic by precomposition with a linear contraction of the simplex.
  
  Now, relative homotopy classes of extensions to $p$ of $g_k|_{\partial p}$, where $p$ is a $(k+1)$-simplex, have a free action by $\pi_{k+1}(Y)$; in particular, differences between them can be labeled by elements of $\pi_{k+1}(Y)$ giving the obstruction to homotoping one to the other.  No matter how we extend $H_{k+1}$ over $p \times [0,1]$, this obstruction will be $g_{k+1}|_p-g_k|_p=\langle \delta b, p \rangle$.  We would like to show that we can do so in such a way that $g_{k+1}|_p$ is $C(k+1,Y)(L+1)$-Lipschitz.

  Note first that by assumption we can extend $g_{k+1}|_{\partial p}$ to $p$ via a $C(k+1,Y)(L+1)$-Lipschitz map $u:D^{k+1} \to Y$.  However, this map may be in the wrong homotopy class.  To build the extension we want, we first estimate the size of the obstruction in $\pi_{k+1}(Y)$ to homotoping $u$ to $g_{k+1}|_p$; Lemma \ref{ref-dist} applied to $Y^{(k+1)}$ then implies that this obstruction is represented by a $C(k+1,Y)(L+1)$-Lipschitz
  map $S^{k+1} \to Y$ which we then glue into the original extension to define
  $g_{k+1}|_p$.
  \begin{lem}
    The obstruction above can be written as $\alpha=\sum_i \alpha_i$ where $\alpha_i \in \pi_{k+1}(Y) \cap \Lambda_i$ and its coefficients in terms of a generating set for this subgroup are $O(L^{k+1+i})$.
  \end{lem}
  In other words, it is contained in a subset of $\pi_{k+1}(Y)$ whose elements, by Lemma \ref{ref-dist}, can be represented by $C(k,Y)(L+1)$-Lipschitz map.  The proof is exactly that of Lemma~4--2 in~\cite{PCDF}, except that Proposition~\ref{qLift} (instead of Proposition 3--9 of that paper) is used to give a bound.

  After fixing $g_{k+1}|_p$ for each $(k+1)$-cell $p$, we can extend $H_{k+1}$ to higher-dimensional cells arbitrarily.  The final task is to build a second-order homotopy from $\Phi_k$ to a homotopy $\Phi_{k+1}$ from $\ph$ to $g_{k+1}$ such that
  \[\Dil_1^U((\Phi_{k+1}|_{\mathcal{M}_Y^*(k+1)})|_{X^{(k+1)}})\leq C(k+1,Y)(L+1).\]
  Intuitively, this can be done since $\Phi_k|_{V_{k+1}}$ and $H_{k+1}^*|_{V_{k+1}}$ have, by construction, very similar integrals over $(k+1)$-cells; hence the obstruction to constructing such a homotopy is easy to kill.  The details are once again the same as in the proof of the shadowing principle in \cite{PCDF}.
\end{proof}

\bibliographystyle{amsalpha}
\bibliography{liphom}

\end{document}